\documentclass[11pt]{amsart}
\usepackage[T1]{fontenc}
\usepackage{lmodern,microtype}
\usepackage{amsmath,amssymb,mathtools,mathrsfs}
\usepackage{booktabs,array}
\usepackage{hyperref,aliascnt}
\usepackage{orcidlink}
\usepackage[nameinlink,capitalise,noabbrev]{cleveref}
\hypersetup{hidelinks,
  pdftitle={The fourth continuous bounded cohomology of the complex symplectic group},
  pdfauthor={Sixuan Gu, Yaoyu Cheng, Wei Qi}}
\allowdisplaybreaks
\numberwithin{equation}{section}
\newtheorem{theorem}{Theorem}[section]
\newaliascnt{proposition}{theorem}
\newtheorem{proposition}[proposition]{Proposition}
\aliascntresetthe{proposition}
\newaliascnt{lemma}{theorem}
\newtheorem{lemma}[lemma]{Lemma}
\aliascntresetthe{lemma}
\newaliascnt{corollary}{theorem}
\newtheorem{corollary}[corollary]{Corollary}
\aliascntresetthe{corollary}
\theoremstyle{definition}
\newaliascnt{definition}{theorem}

\aliascntresetthe{definition}
\newaliascnt{remark}{theorem}

\aliascntresetthe{remark}
\DeclareMathOperator{\Sp}{Sp}
\DeclareMathOperator{\SL}{SL}
\DeclareMathOperator{\SO}{SO}
\DeclareMathOperator{\Pf}{Pf}
\DeclareMathOperator{\Alt}{Alt}
\DeclareMathOperator{\sgn}{sgn}
\DeclareMathOperator{\diag}{diag}
\DeclareMathOperator{\Hom}{Hom}
\DeclareMathOperator{\Stab}{Stab}
\newcommand{\CC}{\mathbb C}
\newcommand{\RR}{\mathbb R}
\newcommand{\PP}{\mathbb P}
\newcommand{\TT}{\mathcal T}
\newcommand{\cU}{\mathcal U}
\newcommand{\Hcb}{H_{\mathrm{cb}}}
\newcommand{\Hc}{H_{\mathrm c}}
\newcommand{\Hm}{H_{\mathrm m}}
\newcommand{\Hmb}{H_{\mathrm{mb}}}

\newcommand{\orb}{\mathrm{orb}}
\newcommand{\eff}{\mathrm{eff}}
\newcommand{\TV}{\mathrm{TV}}
\newcommand{\cochain}{\mathscr C}
\newcommand{\tuplechain}{\widetilde{\mathscr C}}
\newcommand{\cR}{\mathscr R}
\newcommand{\sC}{\mathcal C}
\newcommand{\sM}{\mathcal M}
\title[Vanishing of $H_{\mathrm{cb}}^4(\Sp(4,\CC))$]
{The Fourth Continuous Bounded Cohomology of the Complex Symplectic Group}
\author{Sixuan Gu}
\address{Hetao Institute of Mathemtatics and Interdisciplinery Sciences, Shenzhen, China}
\email{sixuangu@link.cuhk.edu.cn}
\author{Yaoyu Cheng}
\address{Faculty of Business Administration, The Chinese University of Hong Kong,
  Hong Kong, China}
\email{yaoyu@u.nus.edu}
\author[Wei Qi]{Wei Qi\,\orcidlink{0009-0004-5794-4907}}
\address{Department of Physics, The Ohio State University,
  Columbus, OH 43210, USA}
\email{qi.673@osu.edu}
\date{\today}
\begin{document}
\begin{abstract}
We prove that $H_{\mathrm{cb}}^4(\Sp(4,\CC);\RR)=0$. Together with
Blatz's secondary stability and the rank-one theorem of Bucher--Savini,
this gives degree-four vanishing for all complex symplectic and odd
complex orthogonal groups. In normalized symplectic Gram coordinates,
we establish a bounded-primitive estimate and compute the measurable
cohomology of the projective action. An explicit rational cocycle has
divergent periods on a family of finite orbit cycles of uniformly
bounded $\ell^1$-mass. A two-cone averaging construction extends the
period estimate to measurable cochains and excludes bounded
representatives of every nonzero degree-four measurable action class.
\end{abstract}
\maketitle
\setcounter{tocdepth}{1}
\tableofcontents

\section{Introduction}\label{sec:introduction}

Let $V=\CC^4$ carry a nondegenerate alternating bilinear form $\omega$, and
put
\[
 G=\Sp(V,\omega)\cong\Sp(4,\CC),\qquad X=\PP(V).
\]
We regard $G$ as a real Lie group. The projective space $X$ consists of
complex lines in $V$; its measure class is the class of Fubini--Study volume.
The theorem concerns continuous bounded group cohomology with trivial real
coefficients:

\begin{theorem}\label{thm:main}
One has $\Hcb^4(\Sp(4,\CC);\RR)=0$. Consequently, for every $r\ge1$,
\[
 \Hcb^4(\Sp(2r,\CC);\RR)=0,
 \qquad
 \Hcb^4(\SO(2r+1,\CC);\RR)=0.
\]
\end{theorem}

The continuous cohomology is given by the compact-dual calculation
\begin{equation}\label{eq:continuous}
 \Hc^\bullet(G;\RR)
 \cong H^\bullet(\Sp(2);\RR)
 \cong\Lambda_{\RR}(\xi_3,\xi_7).
\end{equation}
Here $\Sp(2)$ denotes the compact quaternionic symplectic group of rank
two, whereas $\Sp(4,\CC)$ is indexed by its complex matrix size. The
generators have the indicated degrees, so $\Hc^4(G;\RR)=0$. The first identification
is the van Est isomorphism with the compact dual of $G/\Sp(2)$, and the
second is the cohomology calculation for compact symplectic groups
\cite{BorelWallach,MimuraToda}. Therefore the rank-two assertion of
\Cref{thm:main} is equivalent to injectivity of the comparison map
$\Hcb^4(G)\to\Hc^4(G)$.

\subsection{Invariant measurable cochains}
The homogeneous coboundary on functions of projective configurations is
\begin{equation}\label{eq:delta}
 (dc)(x_0,\ldots,x_{q+1})
 =\sum_{i=0}^{q+1}(-1)^i
 c(x_0,\ldots,\widehat{x_i},\ldots,x_{q+1}).
\end{equation}
We use the invariant measurable cochain complexes
\begin{equation}\label{eq:action-complexes}
 C_{\mathrm{mb}}^q(X)=L^\infty(X^{q+1})^G,
 \qquad
 C_{\mathrm m}^q(X)=L^0(X^{q+1})^G.
\end{equation}
Here $L^0$ denotes almost-everywhere finite measurable functions modulo
equality almost everywhere, and $L^\infty$ denotes essentially bounded
functions. Invariance is understood at the level of equivalence classes.
The unreduced cohomology spaces $\ker d/\operatorname{im}d$
are denoted by $\Hmb^q(G\curvearrowright X)$ and
$\Hm^q(G\curvearrowright X)$. The inclusion $L^\infty\subset L^0$ induces
the \emph{action comparison map} $c_X^q$. A cochain $f$ with $df=c$ is
called a primitive of $c$. For groups,
$\Hcb^q$ and $\Hc^q$ are defined by the same deletion differential on
invariant continuous functions on $G^{q+1}$, with and without the
boundedness requirement.

The required group-to-action identification and restriction maps are
consequences of the following external results.

\begin{proposition}[Projective reduction and stability]\label{prop:reduction}
There is a linear isomorphism
\begin{equation}\label{eq:projective-reduction}
 \Hcb^4(G;\RR)\cong\Hmb^4(G\curvearrowright X;\RR).
\end{equation}
For $r\ge2$, block inclusions induce injections
\[
 \begin{aligned}
 \Hcb^4(\Sp(2r+2,\CC))&\longrightarrow\Hcb^4(\Sp(2r,\CC)),\\
 \Hcb^4(\SO(2r+3,\CC))&\longrightarrow\Hcb^4(\SO(2r+1,\CC)).
 \end{aligned}
\]
\end{proposition}
\begin{proof}
Set $G_r=\Sp(2r,\CC)$ and $P_r=\PP^{2r-1}(\CC)$. Every complex line
is isotropic for an alternating form. Blatz's
\cite[Lemma~6.2.8, pp.~115--116]{Blatz} therefore gives
\[
 \ker\bigl(\Hcb^4(G_{r+1})\to\Hcb^4(G_r)\bigr)
 \cong\Hmb^4(G_{r+1}\curvearrowright P_{r+1}),\qquad r\ge1.
\]
For $r=1$, the rank-one input
$\Hcb^4(\SL(2,\CC))=0$ follows from Bucher--Savini's degree-four
comparison theorem for real-hyperbolic isometry groups
\cite[Theorem~5]{BucherSavini25}, also recorded in
\cite[Theorem~2.5.5, p.~57]{Blatz}: use
$\mathrm{PSL}(2,\CC)\cong\operatorname{Isom}^{\circ}(\mathbb H^3)$,
$\Hc^4(\SL(2,\CC))=H^4(S^3;\RR)=0$, and invariance of bounded
cohomology under finite central quotients \cite{MonodBook}.
This proves \eqref{eq:projective-reduction}. The injections are
\cite[Theorem~D and Corollary~A, p.~13]{Blatz}.
\end{proof}

The point stabilizer has a Levi subgroup
$\CC^\times\times\SL(2,\CC)$, with nonamenable semisimple factor
$\SL(2,\CC)$. Thus the projective action is not amenable.
The degree-four identification \eqref{eq:projective-reduction} is the
$d_5$ transgression in Blatz's bounded projective spectral sequence;
the $d_2$ in \Cref{sec:ordinary} belongs to the measurable double complex.
The restriction injections in \Cref{prop:reduction} start above rank two.

\subsection{Outline of the proof}
By \Cref{prop:reduction}, it suffices to prove
$\Hmb^4(G\curvearrowright X)=0$. Consider the comparison map
\[
 \Hmb^4(G\curvearrowright X)
 \xrightarrow{\ c_X^4\ }
 \Hm^4(G\curvearrowright X).
\]
Its injectivity follows from the estimate
\begin{equation}\label{eq:intro-bound}
 \|f\|_\infty\le7\|Df\|_\infty
\end{equation}
for alternating measurable functions $f$ on the four-point quotient;
$D$ is the five-term operator defined in \Cref{sec:coordinates}.

The measurable target satisfies
\begin{equation}\label{eq:intro-target}
 \Hm^4(G\curvearrowright X;\RR)\cong
 \Hom_{\RR}(\CC,\RR)\cong\RR^2.
\end{equation}
The real and imaginary parts of an explicit rational cocycle $\cR$
represent a basis. Its periods on a family of finite orbit cycles have a
second-order pole, whereas the cycles have uniformly bounded
$\ell^1$-mass. The period estimate of \Cref{thm:period-test} therefore
excludes bounded representatives of every nonzero class in
$\Hm^4(G\curvearrowright X)$.

Regularization makes the special-configuration identities available
pointwise in the proof of \eqref{eq:intro-bound}. For the period estimate,
a two-cone construction replaces finite orbit cycles by absolutely
continuous configuration measure cycles, with controlled total variation.

\Cref{sec:coordinates} develops the Gram model, and
\Cref{sec:bounded} proves the bounded-primitive estimate.
\Cref{sec:ordinary} computes the measurable target. The rational cocycle,
orbit cycles, and averaging construction occupy
\Cref{sec:cocycle,sec:period-test}. \Cref{sec:completion} proves the
vanishing theorem; the period calculation is recorded in the appendix.

\subsection{Measure conventions}
A set is \emph{conull} if its complement has measure zero. On each smooth
manifold we use the measure class of a positive smooth density; on a
product we use the product class. A map is \emph{nonsingular} if the
inverse image of every null set is null. Pullback along such a map is
well defined on $L^0$ and $L^\infty$ classes. Smooth submersions are
nonsingular, by Fubini in local projection coordinates. All manifolds and
Lie groups below are second countable. A standard measure space here is
a standard Borel space with a nonzero $\sigma$-finite measure class;
we choose an equivalent probability measure when forming $L^0$.
All cochains and chains have real coefficients unless complex values
are explicitly specified. Every cohomology group is unreduced.

\section{Symplectic Gram coordinates}\label{sec:coordinates}

Normalized symplectic Gram matrices describe the generic configuration
orbits and identify invariant measurable cochains with functions on the
corresponding smooth quotients.

\subsection{Generic configurations and normalization}
For $n\ge4$, let $\Omega_n\subset X^n$ consist of ordered tuples
$(\ell_0,\ldots,\ell_{n-1})$ such that $\omega(v_i,v_j)\ne0$ for every $i\ne j$ and
every four lines span $V$, where $v_i\in\ell_i\setminus\{0\}$. For $n=3$ require
pairwise nonzero pairings and linear independence; for $n=2$ require a
nonzero pairing; put $\Omega_1=X$. These are nonempty Zariski-open,
conull sets, preserved by deleting vertices. We call their tuples
\emph{generic}.

Choose nonzero lifts $v_i\in\ell_i$ and form the skew-symmetric matrix
$A=(a_{ij})$, where $a_{ij}=\omega(v_i,v_j)$. Changing the lifts replaces
$A$ by $DAD$ for an invertible diagonal matrix $D$. If $n\ge4$, the
matrix has rank four, all its off-diagonal entries are nonzero, and all
principal $4\times4$ submatrices are nonsingular. Conversely these
conditions characterize Gram matrices of generic tuples.

On every such diagonal-congruence class there is a unique normalized
matrix satisfying
\begin{equation}\label{eq:normalization}
 a_{0j}=1\quad(1\le j<n),\qquad a_{12}=1.
\end{equation}
Indeed, normalizing factors satisfy
$d_j=(d_0a_{0j})^{-1}$ and $d_0^2=a_{12}/(a_{01}a_{02})$.
The two choices of square root differ by a common sign and produce the
same matrix. Denote the normalized locus by $\cU_n$ and the resulting
map by $\gamma_n:\Omega_n\to\cU_n$.

\begin{lemma}[Descent along a translation factor]\label{lem:descent}
Let $J$ be a second countable locally compact group and $Y$ a standard
measure space. For the left-translation action on $J\times Y$,
\[
 L^0(J\times Y)^J\cong L^0(Y).
\]
The identification is pullback by projection, is a homeomorphism for
convergence in measure, and is isometric on $L^\infty$. The same conclusion holds on a countable cover by equivariant product
charts over a common base with trivial action.
\end{lemma}
\begin{proof}
Choose a jointly Borel representative $F$ of an invariant class and
put $u=F/(1+|F|)$. For each $h\in J$,
$u(hg,y)=u(g,y)$ for almost every $(g,y)$. Joint measurability and Fubini
make this an almost-everywhere equality in $(h,g,y)$. The change of
variables $(h,g)\mapsto(hg,g)$ preserves the product Haar measure class,
so $u(k,y)=u(g,y)$ for almost every $(k,g,y)$. Hence, for almost every
$y$, the function $u(\,\cdot\,,y)$ is essentially constant.

Integrate $u$ in the first variable against a fixed Haar-equivalent
probability measure. This gives a measurable essential constant $c(y)$.
It lies in $(-1,1)$ almost everywhere, and
$F(g,y)=c(y)/(1-|c(y)|)$ almost everywhere. This proves surjectivity;
injectivity is Fubini. With product probability measures, pullback
preserves both the metric $\int\min(1,|f-g|)$ and the essential supremum.
For a countable equivariant local-product cover, the descended functions
agree almost everywhere on overlaps. Partitioning the base into
disjoint measurable pieces and using weighted sums of probability
measures gives the global identifications.
\end{proof}

\begin{proposition}[Gram quotient]\label{prop:coordinates}
For $n\ge4$, the fibres of $\gamma_n$ are exactly the $G$-orbits. The
space $\cU_n$ is a smooth Zariski-open subset of $\CC^{3n-10}$, and
$\gamma_n$ is a smooth submersion admitting local product descriptions
\begin{equation}\label{eq:product-chart}
 \gamma_n^{-1}(W)\cong G_{\eff}\times W,
 \qquad G_{\eff}=G/\{\pm I\}.
\end{equation}
These descriptions preserve smooth measure classes. Consequently,
\[
 L^0(\cU_n)\cong L^0(X^n)^G,
 \qquad L^\infty(\cU_n)\cong L^\infty(X^n)^G,
\]
with the latter identification isometric. Deletion induces smooth rational
submersions $\partial_i:\cU_{n+1}\to\cU_n$; permutations induce smooth
rational automorphisms. Both act by nonsingular pullback on cochains.
\end{proposition}
\begin{proof}
Suppose first that two normalized Gram matrices agree. After rescaling
the lifts, the first four vectors in each tuple have identical Gram
matrices. The unique linear map between these ordered bases is
symplectic. Every further vector is determined by its pairings with the
basis, so the same map carries the whole first tuple to the second.
Conversely, a symplectic map preserves the Gram matrix. Every skew matrix
with the stated conditions is realized by taking the images of the
standard basis of $\CC^n$ in the four-dimensional symplectic quotient
$\CC^n/\ker A$. This proves the orbit classification and surjectivity.

Write a normalized matrix in blocks as
\[
 A=\begin{pmatrix}A_0&B\\-B^{\mathsf T}&C\end{pmatrix},
 \qquad b_j=(1,a_{1j},a_{2j},a_{3j})^{\mathsf T}\quad(j\ge4),
\]
where $A_0$ is the leading $4\times4$ block. The rank-four condition is
exactly $C=-B^{\mathsf T}A_0^{-1}B$. Thus $A$ is determined by the two
free entries of $A_0$ and three entries for each further column.
The remaining nonvanishing conditions are open algebraic conditions.
This gives the asserted smooth structure and dimension.

Fix a matrix $J$ for $\omega$. The map $S\mapsto S^{\mathsf T}JS$
from $\mathrm{GL}_4(\CC)$ to nonsingular skew matrices is a surjective
submersion, hence has local smooth sections. Choose such an $S(A_0)$.
Its columns, followed by the vectors $S(A_0)A_0^{-1}b_j$, realize $A$:
their further pairings are $-b_j^{\mathsf T}A_0^{-1}b_k$.
Projectivization gives a local section $\sigma:W\to\Omega_n$.

The effective action on $\Omega_n$ is free. Indeed, if $gv_i=\lambda_i
v_i$ on four generic lines, then all $\lambda_i\lambda_j=1$, so all
$\lambda_i$ are the same sign. Consider
$(g,A)\mapsto g\sigma(A)$. It is bijective by the orbit classification.
Its differential is injective: applying $d\gamma_n$ first kills any
possible base component, and freeness then kills the group component.
Both manifolds have complex dimension $10+(3n-10)=3n$, so the map is a
bijective local diffeomorphism. This proves \eqref{eq:product-chart}.

A diffeomorphism preserves smooth measure classes. In these charts the
action is translation on the first factor; \Cref{lem:descent}, followed
by a countable chart cover, proves the assertions about function spaces.
The complement of $\Omega_n$ in $X^n$ is null and can be discarded.

Finally, deletion $\delta_i:\Omega_{n+1}\to\Omega_n$ is a submersion
and satisfies
\begin{equation}\label{eq:face-naturality}
 \gamma_n\delta_i=\partial_i\gamma_{n+1}.
\end{equation}
The left side is submersive, so $d\partial_i$ is surjective. All
normalized entries are ratios of pairings in which the square roots
cancel, proving rationality. Permutation and normalization give the
assertion for permutations. Nonsingularity follows from submersivity.
\end{proof}

\subsection{Four- and five-point quotients}
The Pfaffian of a $4\times4$ skew matrix is
$a_{01}a_{23}-a_{02}a_{13}+a_{03}a_{12}$; its square is the determinant.
For a normalized quadruple we therefore obtain
\begin{equation}\label{eq:U4}
 A(x,y)=\begin{pmatrix}
 0&1&1&1\\-1&0&1&x\\-1&-1&0&y\\-1&-x&-y&0
 \end{pmatrix},\qquad
 \cU_4=\{(x,y):xy(1-x+y)\ne0\}.
\end{equation}
For an unnormalized four-point Gram matrix $Q$, the quotient coordinates are
\begin{equation}\label{eq:chi}
 \chi(Q)=\left(\frac{q_{13}q_{02}}{q_{12}q_{03}},
                    \frac{q_{23}q_{01}}{q_{12}q_{03}}\right).
\end{equation}

For five points the normalized matrix is
\begin{equation}\label{eq:B}
 B(u,v,w,x,y)=\begin{pmatrix}
 0&1&1&1&1\\-1&0&1&u&v\\-1&-1&0&w&x\\
 -1&-u&-w&0&y\\-1&-v&-x&-y&0
 \end{pmatrix}.
\end{equation}
The space $\cU_5$ is the open subset of $(\CC^\times)^5$ where
\begin{equation}\label{eq:five-pfaffians}
 y-ux+vw,\quad y-x+w,\quad y-v+u,\quad x-v+1,\quad w-u+1
\end{equation}
are all nonzero. These are the five principal $4\times4$ Pfaffians, in
the order of the deleted vertex. Their four-point quotient coordinates
are
\begin{equation}\label{eq:five-faces}
 \left(\frac{ux}{vw},\frac{y}{vw}\right),\quad
 \left(\frac{x}{w},\frac{y}{w}\right),\quad
 \left(\frac vu,\frac yu\right),\quad(v,x),\quad(u,w).
\end{equation}
It follows that the coboundary from degree three to degree four is
\begin{equation}\label{eq:D}
 \begin{split}
 (Df)(u,v,w,x,y)={}&f\left(\frac{ux}{vw},\frac{y}{vw}\right)
 -f\left(\frac{x}{w},\frac{y}{w}\right)
 +f\left(\frac vu,\frac yu\right)\\
 &-f(v,x)+f(u,w).
 \end{split}
\end{equation}
The equality $Df=0$ in $L^0(\cU_5)$ means that this five-term
identity holds almost everywhere. For continuous $f$, it then holds
pointwise on $\cU_5$.

\subsection{The six-point Pfaffian relation}
Write the normalized six-point Gram matrix as
\begin{equation}\label{eq:six-matrix}
 C=\begin{pmatrix}
 0&1&1&1&1&1\\-1&0&1&u&v&s\\-1&-1&0&w&x&t\\
 -1&-u&-w&0&y&z\\-1&-v&-x&-y&0&r\\-1&-s&-t&-z&-r&0
 \end{pmatrix}.
\end{equation}
The space $\cU_6$ is given by $\Pf C=0$, nonzero off-diagonal entries,
and nonzero principal $4\times4$ Pfaffians. Explicitly,
\begin{equation}\label{eq:Pi}
 \begin{split}
 \Pf C={}&-ru+rw+r-sw+sx-sy+tu-tv+ty\\
         &-ux+vw+vz-xz+y-z.
 \end{split}
\end{equation}
The coefficient $w-u+1$ of $r$ is nonzero, so this is a smooth
open subset of $\CC^8$ after eliminating $r$. The next coboundary is
\begin{equation}\label{eq:E}
 \begin{split}
 (EF)(C)={}&F\left(\frac{ux}{vw},\frac{ut}{ws},
                  \frac{y}{vw},\frac{z}{ws},\frac{ru}{vws}\right)\\
 &-F\left(\frac{x}{w},\frac{t}{w},\frac{y}{w},\frac{z}{w},\frac{r}{w}\right)
 +F\left(\frac vu,\frac su,\frac yu,\frac zu,\frac ru\right)\\
 &-F(v,s,x,t,r)+F(u,s,w,t,z)-F(u,v,w,x,y).
 \end{split}
\end{equation}
Each summand is the normalized five-point face of $C$. The face identities
imply $ED=0$, and \Cref{prop:coordinates,prop:reduction} give
\begin{equation}\label{eq:exact-model}
 \Hcb^4(G)\cong\Hmb^4(G\curvearrowright X)
 \cong\frac{\ker(E:L^\infty(\cU_5)\to L^\infty(\cU_6))}
 {D(L^\infty(\cU_4))}.
\end{equation}
Nonsingularity of the face maps makes these operators well defined on
measurable equivalence classes.

\section{Bounded primitives and comparison injectivity}\label{sec:bounded}

The following estimate for $D$ implies injectivity of the action
comparison map.

\begin{theorem}[Bounded primitives]\label{thm:bounded-primitives}
Let $f\in L^0(\cU_4)$ be alternating under permutation of the four
projective points. If $Df\in L^\infty(\cU_5)$, then
\begin{equation}\label{eq:seven}
 \|f\|_\infty\le7\|Df\|_\infty.
\end{equation}
If $f$ is continuous, then the sharper estimate
$\|f\|_\infty\le\|Df\|_\infty$ holds.
\end{theorem}

For continuous primitives, special configurations yield doubling
identities, while finite correspondences reduce their iterates to a
compact set with a polynomial error bound. Fibrewise regularization
then extends the estimate to measurable primitives.

\subsection{Alternation in Gram coordinates}
For a cochain of $m$ variables, define
\[
 \Alt_m c(x_0,\ldots,x_{m-1})
 =\frac1{m!}\sum_{\sigma\in S_m}\sgn(\sigma)
 c(x_{\sigma(0)},\ldots,x_{\sigma(m-1)}).
\]
This projection commutes with $d$ on $L^0$ and $L^\infty$, and has
operator norm at most one on $L^\infty$. Its image is the subcomplex of alternating cochains.

\begin{samepage}
\begin{lemma}[Degree-four alternation]\label{lem:bounded-alternation}
The inclusion of alternating bounded action cochains induces an isometric
isomorphism in degree four. More explicitly, for every bounded action
four-cocycle $F$ there is a bounded invariant three-cochain $b_{\mathrm N}$
such that
\[
 F-\Alt_5F=db_{\mathrm N},\qquad \Alt_4b_{\mathrm N}=0.
\]
Thus the degree-four cohomology of the non-alternating subcomplex
$\ker\Alt$ vanishes.
\end{lemma}
\end{samepage}
\begin{proof}
This is a consequence of Blatz's projective transgression
\cite[Section~5.2 and proof of Lemma~6.2.8, pp.~115--116]{Blatz}.
Use the bounded double complex
\[
 \mathcal B^{p,q}=L^\infty(G^{q+1}\times X^p)^G,
 \qquad p,q\ge0,\qquad X^0=\{*\},
\]
where the index $p$ counts projective vertices. The vertical differential
deletes group variables, and the horizontal differential is $(-1)^q$
times deletion of projective variables. Taking vertical cohomology first
gives
\[
 E_1^{p,q}=\Hcb^q(G;L^\infty(X^p)),\qquad
 E_2^{5,0}=\Hmb^4(G\curvearrowright X).
\]
For $G=\Sp(4,\CC)$, the proof of the cited lemma identifies
\[
 \begin{aligned}
 E_5^{0,4}&=\ker\bigl(\Hcb^4(G)\longrightarrow
                         \Hcb^4(\SL(2,\CC))\bigr),\\
 E_5^{5,0}&=E_2^{5,0},
 \end{aligned}
\]
and proves that
\[
 d_5^{0,4}:E_5^{0,4}\xrightarrow{\ \cong\ }E_5^{5,0}
\]
is an isomorphism.

Alternation in the $p$ projective variables, with $\Alt_0=1$, commutes
with both differentials and hence acts on this spectral sequence. It is
the identity on column zero and therefore on $E_5^{0,4}$. On
$E_5^{5,0}=E_2^{5,0}$ its action is induced by $\Alt_5$. Naturality gives
\[
 (\Alt_5)_*\,d_5^{0,4}=d_5^{0,4}.
\]
Surjectivity of the transgression implies $(\Alt_5)_*=1$ on degree-four
bounded action cohomology. Inclusion and alternation are consequently
inverse in that degree. Both have norm at most one, so the isomorphism
is isometric for the quotient seminorms.

The idempotent cochain map $\Alt$ gives the direct sum of complexes
$C_{\mathrm{mb}}^\bullet(X)=\operatorname{im}\Alt\oplus\ker\Alt$.
The identity just proved therefore implies $H^4(\ker\Alt)=0$.
Apply this to the cocycle $F-\Alt_5F$ in $\ker\Alt$ to obtain the
stated bounded invariant primitive $b_{\mathrm N}$.
\end{proof}

Only the degree-four reduction is used below. For finite measurable
cochains we use the identity $d\Alt=\Alt d$, without a cohomological
alternation equivalence for the $L^0$ complex.

In the coordinates \eqref{eq:U4}, the permutation action factors through
$S_4/V_4\cong S_3$, where the Klein four group $V_4$ acts trivially.
Permuting the Gram matrix and applying \eqref{eq:chi} gives the table
\begin{equation}\label{eq:symmetries}
\begin{array}{c|ccc}
\text{even}&(x,y)&(-1/y,-x/y)&(y/x,-1/x)\\[2pt]
\text{odd}&(1/x,-y/x)&(-y,-x)&(x/y,1/y).
\end{array}
\end{equation}
Thus coordinate alternation is the signed average over these six maps.
Set
\begin{equation}\label{eq:C-rho}
 \sC(x,y)=(-1/y,-x/y),\qquad \rho(x,y)=(-y,-x).
\end{equation}
For pointwise alternating $f$,
$f(\sC z)=f(z)$ and $f(\rho z)=-f(z)$. The fixed loci of the three odd symmetries give
\begin{equation}\label{eq:odd-fixed}
 f(-1,y)=f(x,-x)=f(x,1)=0
\end{equation}
whenever the displayed points belong to $\cU_4$.

\subsection{Special configurations and doubling identities}
Define maps on $\TT=(\CC^\times)^2$ by
\begin{align}
 T_1(a,b)&=(-a,-b^2),& T_2(a,b)&=(ab,-b^2),\label{eq:T}\\
 \Phi_1(a,b)&=(a,-a,b,-b^2,-ab),&
 \Phi_2(a,b)&=(a,ab,b,-b^2,-ab).\label{eq:Phi}
\end{align}
We use them only where the corresponding four- and five-point
configurations are generic.

\begin{lemma}[Doubling identities]\label{lem:dilation}
For a pointwise alternating function $f$,
\begin{equation}\label{eq:dilation}
 (Df)(\Phi_j(z))=2f(z)-f(T_jz),\qquad j=1,2,
\end{equation}
on the generic domain. If $f$ is continuous and
$M=\|Df\|_\infty<\infty$, then
$|2f(z)-f(T_jz)|\le M$ there.
\end{lemma}
\begin{proof}
Using \eqref{eq:five-faces}, the faces in order, with signs
$+,-,+,-,+$, are
\[
\begin{array}{c|ccccc}
 &0&1&2&3&4\\ \hline
\Phi_1(a,b)&(b,1)&(-b,-a)&(-1,-b)&T_1(a,b)&(a,b)\\
\Phi_2(a,b)&(-1,-1/b)&(-b,-a)&(b,-b)&T_2(a,b)&(a,b).
\end{array}
\]
The first and third terms vanish by \eqref{eq:odd-fixed}, and
$(-b,-a)=\rho(a,b)$, which gives \eqref{eq:dilation}. If $f$ is
continuous, so is $Df$ on $\cU_5$. Its essential bound is then a
pointwise bound, since every nonempty open set has positive smooth
measure.
\end{proof}

Continuity is required here: the image of $\Phi_j$ has positive
codimension in $\cU_5$, so an almost-everywhere bound alone does not
restrict to that image.

\subsection{Finite correspondences and defect estimates}
Fix a continuous alternating function $f$ with
$M=\|Df\|_\infty<\infty$. Subtracting the two identities
\eqref{eq:dilation} gives
\begin{equation}\label{eq:move-cost}
 |f(T_1z)-f(T_2z)|\le2M.
\end{equation}
Let $U$ be the finite correspondence with parameter space
$\mathscr E_U=\TT$ and source and target maps $s=T_1$, $t=T_2$.
Both maps are finite \emph{\'{e}tale} of degree two, so the correspondence
retains both square-root branches of $T_1^{-1}$. Its admissible parameter
locus is defined by the polynomial
\begin{equation}\label{eq:discriminant}
 \begin{gathered}
 \mathscr E_U^\circ=\{(a,b):\Delta_U(a,b)\ne0\},\\
 \Delta_U=ab(b-2)(2b-1)(a-b-1)(a-b^2+1)(1-ab-b^2).
 \end{gathered}
\end{equation}
Substitution into \eqref{eq:five-pfaffians} shows that admissibility
makes $\Phi_1,\Phi_2$ and all their faces generic. A $U$-step is a
pair $s(p)\to t(p)$; reversing its direction gives a $U^{-1}$-step.
Each changes $f$ by at most $2M$, whereas $\sC^{\pm1}$ preserves $f$.
A lift of a word is a sequence of correspondence parameters with matching
endpoints.

\begin{lemma}[Generic lifts of finite words]\label{lem:words}
For each finite word $w$ in $U^{\pm1},\sC^{\pm1},\rho$, there is a
proper closed algebraic subset $B_w\subset\TT$ such that every lift of
$w$ starting outside $B_w$ has all its intermediate points in $\cU_4$ and all its
$U^{\pm1}$ parameters in $\mathscr E_U^\circ$. At least one such lift
exists at every point outside $B_w$.
\end{lemma}
\begin{proof}
Represent $U^{-1}$ by $(\mathscr E_U,t,s)$ and a symmetry by its graph.
If $w=\varepsilon_m\cdots\varepsilon_1$, read right to left and take
the fibre product of the letter correspondences, matching each target
with the next source. Call this variety $\mathscr E_w$ and its source
map $s_w$.

Every projection from $\mathscr E_w$ to a copy of $\TT$ parametrizing
an intermediate point is a finite \'{e}tale surjection, by base change
and composition. The same is true
of projection to any letter's parameter space. To see the latter, split
the word before and after that letter. Its projection is obtained by
adjoining the prefix and suffix through base changes of their endpoint
coverings. All these targets are connected copies of $\TT$.

A connected component of $\mathscr E_w$ is smooth, hence irreducible.
Each of the preceding maps is surjective on that component: its image is
nonempty, open by the \'{e}tale property, and closed by finiteness. Pull
back the divisor $1-x+y=0$ along every intermediate-point map and the divisor
$\Delta_U=0$ along every $U^{\pm1}$ parameter projection. None of these
pullbacks contains a component. Their finite union $\mathscr B_w$
therefore has dimension at most one on every component. Since $s_w$ is
finite, $B_w=s_w(\mathscr B_w)$ is closed of dimension at most one in
the two-dimensional torus. No lift above a point outside $B_w$ meets
$\mathscr B_w$, and surjectivity of $s_w$ gives existence. For the empty
word take $B_w=\{1-x+y=0\}$.
\end{proof}

The nongeneric parameter locus of $\Phi_1$ is
\[
 B_{\mathrm{dil}}=\{(b-2)(a-b-1)(a-b^2+1)=0\}\subset\TT.
\]
Define
\begin{equation}\label{eq:fully-generic}
 \cU_*=
 \TT\setminus\bigcup_{n\ge0}T_1^{-n}
       \left(B_{\mathrm{dil}}\cup\bigcup_w B_w\right).
\end{equation}
This is a conull dense subset of $\cU_4$. Indeed, there are countably
many words, and the finite dominant map $T_1$ pulls a proper algebraic
subset back to a proper algebraic subset. These sets have measure zero;
their complements have dense intersection by Baire's theorem.
The set $\cU_*$ is forward invariant under $T_1$, since avoidance at all
iterates is preserved by shifting the iteration index. At each iterate
of a point in $\cU_*$, every finite word has only admissible lifts.

\subsection{Dyadic shears and compact reduction}
Write
\[
 \mathbf L(x,y)=(\log|x|,\log|y|)^{\mathsf T}.
\]
The signs in \eqref{eq:T} disappear under $\mathbf L$. A $U$-step and
the symmetry $\sC$ act on logarithmic radii by the matrices
\begin{equation}\label{eq:radius-generators}
 \mathbf U=\begin{pmatrix}1&1/2\\0&1\end{pmatrix},
 \qquad
 \mathbf C=\begin{pmatrix}0&-1\\1&-1\end{pmatrix}.
\end{equation}
Thus a correspondence word has a well-defined action on these
logarithmic absolute values, even though it may have several lifts on the complex torus.

\begin{lemma}[Dyadic shears and word length]\label{lem:shears}
Put $u(t)=\left(\begin{smallmatrix}1&t\\0&1\end{smallmatrix}\right)$.
The group generated by $\mathbf U,\mathbf C$ contains every dyadic
shear $u(k/2^m)$. If $q_n=k_n/2^{2n}$ with $|k_n|=O(2^n)$, then
$u(q_n)$ has a word of length $O(n^2)$ in these generators and their
inverses. It also contains
$\mathbf R=\diag(2,1/2)$.
\end{lemma}
\begin{proof}
Direct multiplication gives
\begin{equation}\label{eq:radial-factorization}
 \mathbf V=\mathbf C\mathbf U\mathbf C^{-1}
 =\begin{pmatrix}1&0\\-1/2&1\end{pmatrix},\qquad
 \mathbf R=\mathbf U^4\mathbf V\mathbf U^2\mathbf V^{-2}\mathbf U^{-2}.
\end{equation}
Since $\mathbf R^k u(t)\mathbf R^{-k}=u(4^kt)$, conjugates of
$\mathbf U=u(1/2)$ and $\mathbf U^2=u(1)$ give $u(\pm2^{-m})$ for
all $m\ge0$. A signed binary expansion of $q_n$ has $O(n)$ possible
nonzero digits. Each digit uses a conjugate of length $O(n)$, so the
concatenated word has length $O(n^2)$.
\end{proof}

Bounded logarithmic absolute values confine the endpoints to a compact
product of annuli in $\TT$. The following maps avoid the excluded divisor
$1-x+y=0$ at any limit point.

\begin{samepage}
\begin{lemma}[Avoidance of the degeneracy divisor]\label{lem:correction}
The maps
\[
 S_1(x,y)=(xy,y),\qquad S_2(x,y)=(x,-xy)=\rho S_1\rho(x,y)
\]
are realized, respectively, by two $U$-steps and by their conjugate
under $\rho$. On an admissible realization,
$|f(S_i z)-f(z)|\le4M$. If $1-x+y=0$ on $\TT$, at least one of
$S_1(x,y),S_2(x,y)$ lies in $\cU_4$.
\end{lemma}
\end{samepage}
\begin{proof}
Choose $\beta^2=-y$. The parameters $(-x,\beta)$ and
$(x\beta,-\beta)$ define consecutive steps from $(x,y)$ to $(xy,y)$.
This proves the estimate for $S_1$; conjugating and using
$f\rho=-f$ proves it for $S_2$. If both $1-x+y=0$ and
$1-xy+y=0$, then $x=1+y$ and $1-y^2=0$. Since $x\ne0$, the point
is $(2,1)$, and $1-x-xy=-3$ there. Hence $S_2(2,1)\in\cU_4$.
\end{proof}

\begin{proposition}[Compact reduction]\label{prop:return}
Let $f$ be continuous and alternating, with $M=\|Df\|_\infty<\infty$.
For each $z\in\cU_*$ there are a constant $C_z>0$, integers
$n_k\to\infty$, and points $w_k$ in a compact subset of $\cU_4$ such that
\begin{equation}\label{eq:return-estimate}
 |f(T_1^{n_k}z)-f(w_k)|\le C_z M n_k^2+4M.
\end{equation}
\end{proposition}
\begin{proof}
Write $\mathbf L(z)=(a,b)^{\mathsf T}$, so
$\mathbf L(T_1^nz)=(a,2^nb)^{\mathsf T}$. For $b\ne0$, choose a
dyadic $q_n=k_n/2^{2n}$ such that
\[
 \left|q_n-\frac{a}{2^nb}\right|\le2^{-2n-1}.
\]
Then $|k_n|=O_z(2^n)$ and
\begin{equation}\label{eq:bounded-radii}
 \mathbf R^n u(-q_n)\binom a{2^nb}
 =\binom{2^n(a-q_n2^nb)}b
\end{equation}
is bounded independently of $n$. By \Cref{lem:shears}, the matrix
$\mathbf R^n u(-q_n)$ is represented by a word of length $O_z(n^2)$.
For $b=0$, use the empty word, since the radii are already bounded.

Lift the chosen word starting at $T_1^nz$. \Cref{lem:words} and
\eqref{eq:fully-generic} make every such lift admissible. Let $y_n$
be its endpoint. Each $U^{\pm1}$ step changes $f$ by at most $2M$, and each
$\sC^{\pm1}$ step preserves $f$. Therefore
$|f(T_1^nz)-f(y_n)|\le O_z(Mn^2)$.
By \eqref{eq:bounded-radii}, the points $y_n$ lie in a fixed compact
product of annuli in $\TT$. Pass to a convergent subsequence $y_{n_k}\to q$.

If $q\in\cU_4$, take $w_k=y_{n_k}$. Otherwise fix one of the two
corrections of \Cref{lem:correction} with $S_iq\in\cU_4$ and set
$w_k=S_i y_{n_k}$. Its explicit lift can be appended to the original
word. Apply \Cref{lem:words} to this \emph{concatenated} word starting
at $T_1^{n_k}z\in\cU_*$. Every lift is admissible, including the
chosen first path and correction. The additional change in $f$ is at most $4M$.
For all sufficiently large $k$, the points $w_k$ lie in a compact
neighborhood of $S_iq$ contained in $\cU_4$. Discarding finitely many
terms and choosing $C_z$ to dominate the word-length bound proves the
assertion.
\end{proof}

\begin{proof}[Proof of the continuous estimate]
Iterating \eqref{eq:dilation} for $T_1$ gives, on $\cU_*$,
\[
 |2^nf(z)-f(T_1^nz)|\le(2^n-1)M.
\]
Use the subsequence in \Cref{prop:return}. Continuity bounds $|f(w_k)|$
by a constant $B_z$, so
\[
 |f(z)|\le(1-2^{-n_k})M+
 2^{-n_k}\bigl(B_z+C_z M n_k^2+4M\bigr).
\]
The second term tends to zero. Hence $|f(z)|\le M$ on the dense set
$\cU_*$, and continuity extends the bound to all of $\cU_4$.
\end{proof}

\subsection{Regularization of measurable primitives}
Write $z=(u,w)$ and $p=(v,x,y)$. Projection to $z$ gives a surjective
submersion $\pi:\cU_5\to\cU_4$. To see surjectivity, fix $(u,w)$;
choose nonzero $v,x$ with $x-v+1\ne0$, and then choose $y$ outside
$\{0,ux-vw,x-w,v-u\}$.

Denote the other four faces by
\[
 \theta_0=\left(\frac{ux}{vw},\frac y{vw}\right),\quad
 \theta_1=\left(\frac xw,\frac yw\right),\quad
 \theta_2=\left(\frac vu,\frac yu\right),\quad
 \theta_3=(v,x).
\]
For fixed $z$, each map $p\mapsto\theta_i(z,p)$ is a submersion.
Use the variables $(x,y)$ for the first two, $(v,y)$ for the third,
and $(v,x)$ for the fourth. Equation \eqref{eq:D} becomes
\begin{equation}\label{eq:sliced-D}
 Df(z,p)=\sum_{i=0}^3(-1)^if(\theta_i(z,p))+f(z).
\end{equation}
The regularization is obtained by integrating the four face terms in
\eqref{eq:sliced-D}. Local integrability follows from
\Cref{lem:local-bounded}.

\begin{lemma}[Properly supported fibre integration]\label{lem:kernels}
Let $\pi:M\to B$ be a surjective smooth submersion. There exists a
smooth nonnegative vertical density $\eta$ on $M$ whose restrictions
$\eta_b$ to the fibres have total mass one and whose support is proper
over $B$. If $\theta:M\to Y$ is smooth
and $(\pi,\theta):M\to B\times Y$ is a submersion, then for a fixed
positive smooth density $d\lambda$ on $Y$ there is a smooth kernel
$K(b,y)\ge0$, with support proper over $B$, such that
\begin{equation}\label{eq:kernel}
 \int_{\pi^{-1}(b)}h(\theta(m))\eta_b(m)
 =\int_Y h(y)K(b,y)\,d\lambda(y).
\end{equation}
The identity holds for nonnegative measurable $h$. For
$h\in L^1_{\mathrm{loc}}(Y)$, both sides are finite, the right side is
smooth in $b$, and the value depends only on the almost-everywhere class
of $h$.
\end{lemma}
\begin{proof}
Choose a locally finite base cover by relatively compact charts with
product regions $V_j\times W_j\subset M$. Let $(\chi_j)$ be a
subordinate smooth partition of unity with supports compactly contained
in $V_j$, and choose a smooth probability density $\tau_j$ compactly
supported in $W_j$. Extend $\chi_j\tau_j$ by zero and sum. The resulting
vertical density has fibre integral one. Over a compact subset of $B$
only finitely many compact product supports occur, proving properness.

Push this density forward by $(\pi,\theta)$, using an auxiliary base
density. In submersion coordinates the map is $(b,y,t)\mapsto(b,y)$,
so the pushforward is obtained by integrating a smooth density in $t$.
Properness on the support gives a smooth kernel $K$. Over each compact
base set it has a common compact $y$-support, and all its base
derivatives are bounded. Differentiation under the integral proves the
last assertion. The same construction applies to any already chosen
smooth vertical density with proper support.
\end{proof}

\begin{lemma}[Local essential boundedness]\label{lem:local-bounded}
If $f\in L^0(\cU_4)$ and $Df\in L^\infty(\cU_5)$, then
$f\in L^\infty_{\mathrm{loc}}(\cU_4)$.
\end{lemma}
\begin{proof}
Fix a positive smooth density $\lambda$ on $\cU_4$. Near any base
point choose boxes $Z_0\Subset Z_1\Subset\cU_4$ and
a fibre box $P$ with $\overline{Z_1}\times\overline P\subset\cU_5$.
Choose a smooth probability density $\nu$ compactly supported in $P$.
Since $(z,p)\mapsto(z,\theta_i(z,p))$ is a submersion,
\Cref{lem:kernels} gives smooth pushforward kernels $K_i(z,\zeta)$
for $\nu$. For $z\in\overline{Z_0}$ their supports lie in one compact
set $K\subset\cU_4$, and all four kernels are bounded by a constant
$C$.

Since $f$ is finite almost everywhere,
$\lambda(K\cap\{|f|>L\})\to0$ as $L\to\infty$. Choose $L$ so
large that $4C\lambda(K\cap\{|f|>L\})<1$. For each $z\in Z_0$,
the union bound then gives positive $\nu$-measure to the set of $p$
for which all four $|f(\theta_i(z,p))|\le L$. For almost every $z$,
Fubini also makes \eqref{eq:sliced-D} and $|Df(z,p)|\le M$ valid
for $\nu$-almost every $p$, where $M=\|Df\|_\infty$.
Choosing $p$ in this intersection yields
$|f(z)|\le M+4L$. A countable cover by such $Z_0$ proves the claim.
\end{proof}

\begin{proposition}[Regularization with bounded error]\label{prop:regularization}
If $f\in L^0(\cU_4)$ and $M=\|Df\|_\infty<\infty$, there exists
$Sf\in C^\infty(\cU_4)$ with $\|f-Sf\|_\infty\le M$.
\end{proposition}
\begin{proof}
Choose a vertical probability density $\eta$ for $\pi$ as in
\Cref{lem:kernels}, and define
\[
 Sf(z)=-\int_{\pi^{-1}(z)}
             \sum_{i=0}^3(-1)^if(\theta_i(z,p))\,\eta_z(p).
\]
By \Cref{lem:local-bounded}, $f$ is locally integrable. Each
$(\pi,\theta_i)$ is submersive, so \Cref{lem:kernels} makes $Sf$
smooth and independent of representatives. Fubini in local submersion
charts allows us to integrate \eqref{eq:sliced-D} for almost every
$z$. It gives
\[
 f(z)-Sf(z)=\int_{\pi^{-1}(z)}Df(z,p)\,\eta_z(p).
\]
The density has mass one, so the essential bound is at most $M$.
\end{proof}

\begin{proof}[Completion of the proof of \Cref{thm:bounded-primitives}]
Assume that $f$ is measurable and alternating. Put
$M=\|Df\|_\infty$ and $h=\Alt_4(Sf)$ in quotient coordinates. The six maps in \eqref{eq:symmetries} are
smooth, so $h$ is smooth and alternating. Since $\Alt f=f$,
\[
 \|f-h\|_\infty\le M,
 \qquad
 \|Dh\|_\infty\le\|Df\|_\infty+5\|f-h\|_\infty\le6M.
\]
Apply the continuous case to $h$: $\|h\|_\infty\le6M$.
The triangle inequality gives \eqref{eq:seven}.
\end{proof}

\begin{corollary}[Comparison injectivity]\label{cor:injective}
The map $c_X^4:\Hmb^4(G\curvearrowright X)\to
\Hm^4(G\curvearrowright X)$ is injective.
\end{corollary}
\begin{proof}
Let a bounded class map to zero. By \Cref{lem:bounded-alternation},
choose an alternating bounded representative $F$. There is an invariant
finite measurable primitive $f$ with $F=df$, and then also
$F=d(\Alt_4 f)$. Descend to $\cU_4$ using \Cref{prop:coordinates}.
\Cref{thm:bounded-primitives} makes this alternating primitive bounded.
Thus the original class is zero in bounded action cohomology.
\end{proof}

\section{Measurable cohomology of the projective action}\label{sec:ordinary}

The measurable cohomology of the projective action is determined by
the stabilizers of generic configurations.

\begin{theorem}\label{thm:ordinary}
There is a linear isomorphism
\[
 \Hm^4(G\curvearrowright X;\RR)\cong\Hom_{\RR}(\CC,\RR).
\]
In particular this space has real dimension two.
\end{theorem}

The isomorphism is the $d_2$ transgression from the first cohomology
of the generic triple stabilizer. We compute the stabilizers and the
relevant differentials below.

\subsection{Stabilizers of generic configurations}
Choose a symplectic basis $(e_1,f_1,e_2,f_2)$ with
$\omega(e_i,f_j)=\delta_{ij}$, and put $W=\langle e_2,f_2\rangle$.
Let
\begin{align*}
 Q&=\Stab_G([e_1]),\\
 L&=\Stab_G([e_1],[f_1]),\\
 H&=\Stab_G([e_1],[f_1],[e_1+f_1+e_2]).
\end{align*}
These are stabilizers of ordered configurations; each line is fixed
individually.

\begin{proposition}\label{prop:stabilizers}
The spaces $\Omega_1,\Omega_2,\Omega_3$ are single $G$-orbits, with
stabilizers
\begin{equation}\label{eq:stabilizers}
 Q=L\ltimes N,\qquad
 L\cong\CC^\times\times\SL(2,\CC),\qquad
 H\cong\{\pm1\}\times(\CC,+).
\end{equation}
Here $N$ is the complex three-dimensional Heisenberg group. For $p\ge4$
the stabilizer on $\Omega_p$ is $Z(G)=\{\pm I\}$.
\end{proposition}
\begin{proof}
Transitivity on lines and on symplectically nonorthogonal ordered pairs
follows by
extending a pair of vectors with pairing one to a symplectic basis.
An element of the pair stabilizer sends
$e_1\mapsto ae_1$, $f_1\mapsto a^{-1}f_1$, and preserves $W$, acting
there by some $B\in\SL(W)$. This gives the description of $L$.

After fixing the first two lines, a generic third line has a lift
$\alpha e_1+\beta f_1+w$, where $\alpha\beta\ne0$ and $w\ne0$.
Choose $a$ with $a^2=\beta/\alpha$, and then use transitivity of
$\SL(W)$ on $W\setminus\{0\}$ to obtain the displayed standard
triple. An element $(a,B)\in L$ fixes that triple precisely when
\[
 ae_1+a^{-1}f_1+Be_2=c(e_1+f_1+e_2).
\]
Hence $a=a^{-1}=c=\varepsilon\in\{\pm1\}$, and, in the basis
$(e_2,f_2)$,
\[
 B=\varepsilon\begin{pmatrix}1&t\\0&1\end{pmatrix},\qquad t\in\CC.
\]
The parameter $t$ adds under multiplication and the sign is central.

The point stabilizer acts on the line $\CC e_1$ and on the symplectic
quotient $e_1^\perp/\CC e_1\cong W$. This gives its projection onto
$L$. Its kernel consists of the maps
\begin{equation}\label{eq:Heisenberg}
 \begin{aligned}
 n(w,c)e_1&=e_1,&
 n(w,c)f_1&=f_1+w+ce_1,\\
 n(w,c)v&=v+\omega(w,v)e_1 &&(v\in W),
 \end{aligned}
\end{equation}
with $w\in W$, $c\in\CC$. Direct multiplication gives
\[
 n(w,c)n(w',c')=n(w+w',c+c'+\omega(w,w')),
\]
the Heisenberg group law. The projection has the standard block
splitting, proving $Q=L\ltimes N$. The final assertion was established
in the freeness argument of \Cref{prop:coordinates}.
\end{proof}

\begin{lemma}[Stabilizer cohomology]\label{lem:stabilizer-cohomology}
Let $A=\Hom_{\RR}(\CC,\RR)$, a real two-dimensional vector space.
Then
\begin{equation}\label{eq:stabilizer-rings}
 \Hc^\bullet(L;\RR)\cong\Hc^\bullet(Q;\RR)
 \cong\Lambda(\eta_1,\xi_3),
 \qquad
 \Hc^\bullet(H;\RR)\cong\Lambda A.
\end{equation}
The degrees of $\eta_1,\xi_3$ are one and three. Restriction
$\Hc^3(G)\to\Hc^3(Q)$ is an isomorphism. Each of the three face
restrictions from the cohomology of a pair stabilizer to that of the
triple stabilizer is zero in degree one.
\end{lemma}
\begin{proof}
For $L=\CC^\times\times\SL(2,\CC)$, van Est and the product formula
give the exterior algebra in \eqref{eq:stabilizer-rings}. The degree-one
generator is the continuous real character $(a,B)\mapsto\log|a|$; the degree-three
generator comes from $\Hc^3(\SL(2,\CC))=H^3(S^3;\RR)$
\cite{BorelWallach}. A vector group has continuous cohomology the
exterior algebra on its real dual, and a finite direct factor contributes
no positive-degree real cohomology. The degree-one generators for $H$ are the
homomorphisms $t\mapsto\operatorname{Re}t$ and
$t\mapsto\operatorname{Im}t$.

For $r>1$, the central element
$a_r=\diag(r,r^{-1},1,1)$ of $L$ conjugates $n(w,c)$ to
$n(rw,r^2c)$. Thus its eigenvalues on the real Lie algebra of $N$ are
$r$ and $r^2$. On every positive-degree exterior power of the dual,
and hence on its cohomology, all eigenvalues are $r^{-m}$ with $m>0$.
The van Est identification for the simply connected nilpotent group
$N$ gives the same statement for $\Hc^s(N;\RR)$.

For a locally compact
second countable group $J$ and a finite-dimensional continuous real
$J$-module $E$, if $z\in Z(J)$ and $z-1:E\to E$ is invertible,
then $\Hc^q(J;E)=0$ for every $q$. Its inverse is continuous and
$J$-equivariant. Indeed,
simultaneous multiplication of the arguments by $z$ on the homogeneous
continuous complex acts as $z$ on coefficients. The prism
\[
 (Kc)(g_0,\ldots,g_{q-1})
 =\sum_{i=0}^{q-1}(-1)^i
 c(g_0,\ldots,g_i,zg_i,zg_{i+1},\ldots,zg_{q-1})
\]
with $K=0$ in degree zero satisfies $dK+Kd=z-1$. Centrality makes
$K$ equivariant. Thus $z-1$
induces both zero and an isomorphism on cohomology, forcing vanishing.
The prism acts on continuous cochains.

Apply this fact to $J=L$ and $E=\Hc^s(N;\RR)$ for $s>0$. The
continuous Hochschild--Serre spectral sequence
\[
 \Hc^p(L;\Hc^s(N;\RR))\Longrightarrow\Hc^{p+s}(Q;\RR)
\]
has no positive-$s$ rows. It therefore identifies inflation from $L$
to $Q$ with an isomorphism \cite{HochschildMostow,BorelWallach}.

For restriction in degree three, write $p_Q:Q\to L$ for the projection
and $i_L:L\hookrightarrow Q$ for the splitting. Since $p_Q^*$ is
an isomorphism and $i_L^*p_Q^*=1$, restriction from $G$ to $Q$ is
an isomorphism in degree three exactly when restriction to $L$ is.
The block inclusion $L\hookrightarrow G$ preserves the Cartan involution
$g\mapsto(g^*)^{-1}$. Its maximal compact subgroups are
$K_L=U(1)\times\Sp(1)\subset K_G=\Sp(2)$. On compact duals the
map is induced by $j\times j$, where $j:K_L\hookrightarrow K_G$.
Under $(K\times K)/K_{\mathrm{diag}}\cong K$ this is precisely $j$.
Naturality of van Est therefore identifies restriction with $j^*$.
The map $H^3(\Sp(2))\to H^3(\Sp(1))$ is an isomorphism by the
fibration $\Sp(1)\to\Sp(2)\to S^7$, and the $U(1)$ factor adds
nothing to degree three \cite{MimuraToda}. This proves the assertion.

Finally an element $(\varepsilon,t)\in H$ acts on each of the three
standard lines by the same sign $\varepsilon$. Its scalar coordinate in
any of the three pair stabilizers is $\varepsilon$, possibly inverted.
The degree-one character restricts to $\log|\varepsilon|=0$.
Thus every face restriction in degree one is zero.
\end{proof}

\subsection{Measurable induction and naturality}
For a standard measure space $Y$, choose a probability measure $\mu$
in its measure class. Convergence in measure on $L^0(Y)$ is metrized by
\[
 d_0(f,h)=\int_Y\min(1,|f-h|)\,d\mu
\]
and makes $L^0(Y)$ a Polish vector space. The topology is independent
of the chosen probability representative. Fubini gives the exponential law
$L^0(Y;L^0(Z))=L^0(Y\times Z)$.

The smooth action on compact $X^p$ induces a continuous action on
$L^0(X^p)$: approximation by continuous functions in $d_0$ and uniform
boundedness of the Radon--Nikodym derivatives near the identity prove
continuity. Thus $L^0(X^p)$ is a Polish $G$-module for Moore cohomology.
For a Polish $J$-module $E$, write $\Hm^q(J;E)$ for that theory,
computed by equivariant measurable homogeneous cochains modulo null
sets. Here equivariance means
$c(hg_0,\ldots,hg_q)=h\,c(g_0,\ldots,g_q)$, and the differential
still deletes arguments. For a function-space coefficient, the action is
$(hF)(\mathbf x)=F(h^{-1}\mathbf x)$.
For finite-dimensional real coefficients, it agrees naturally with
continuous cohomology by Austin--Moore \cite[Theorem~A]{AustinMoore}.
The $L^0$-coefficient columns are treated by measurable induction.

\begin{lemma}[Measurable induction and face maps]\label{lem:shapiro}
Let $K<G$ be closed, let $Y$ be a standard measure space with trivial
$G$-action, and give $A_Y=L^0(Y)$ the trivial $K$-action. Then
\begin{equation}\label{eq:shapiro}
 \Hm^q(G;L^0(Y\times G/K))\cong\Hm^q(K;A_Y).
\end{equation}
If $K$ is finite, these groups vanish for $q>0$.
For closed $H,K<G$ and $a^{-1}Ha\subset K$, the pullback induced by
\[
 G/H\longrightarrow G/K,\qquad gH\longmapsto gaK
\]
corresponds, with real coefficients, to restriction along
$H\to K$, $h\mapsto a^{-1}ha$.
\end{lemma}
\begin{proof}
Give $L^0(G/K;A_Y)$ the induced Polish module topology, which the
exponential law identifies with convergence in measure on
$L^0(Y\times G/K)$. Equation \eqref{eq:shapiro} is Moore's
measurable induction theorem for Polish coefficients
\cite[Theorem~6]{Moore}; see also
\cite[Proposition~2.2]{MonodBoundary}. The homogeneous complex on the ambient group $G$,
\[
 C_G^q(K;A_Y)=L^0(G^{q+1}\times Y)^K
\]
computes $\Hm^q(K;A_Y)$; this is the homogeneous realization of
Moore's characterization theorem
\cite[Theorem~2]{Moore}, discussed in \cite[Section~4]{MonodBoundary}.
All stabilizers are thus treated on the same ambient space $G^{q+1}$.

By local sections of $G\to K\backslash G$ and \Cref{lem:descent},
a source class has a jointly measurable representative $\varphi^\sharp$
that is pointwise $K$-invariant. The induction map is
\begin{equation}\label{eq:induction-formula}
 (I_K\varphi)(g_0,\ldots,g_q)(y,xK)
 =\varphi^\sharp(x^{-1}g_0,\ldots,x^{-1}g_q;y).
\end{equation}
To verify that the formula is defined on equivalence classes, identify
the source quotient with
\[
 (Kg_0,g_0^{-1}g_1,\ldots,g_0^{-1}g_q;y),
\]
and the target quotient with
\[
 (g_0^{-1}g_1,\ldots,g_0^{-1}g_q;y,g_0^{-1}xK).
\]
The two quotient spaces are identified by inversion
$G/K\to K\backslash G$, $zK\mapsto Kz^{-1}$, and reordering the
factors. Local product charts show that this is a measure-class
isomorphism. It gives \eqref{eq:induction-formula} on pointwise invariant
representatives and proves representative-independence and bijectivity.
Deletion of any $G$-variable commutes with this formula.

For the last assertion, the corresponding ambient cochain map is
\begin{equation}\label{eq:ambient-face}
 (R_a\varphi)(g_0,\ldots,g_q)
 =\varphi(a^{-1}g_0,\ldots,a^{-1}g_q).
\end{equation}
Translation is nonsingular, and $a^{-1}Ha\subset K$ gives the required
invariance. The homogeneous-space map $\Theta_a:gH\mapsto gaK$ is
a smooth submersion, hence is nonsingular as well. On pointwise invariant
representatives, \eqref{eq:induction-formula} gives
\[
  \Theta_a^* I_K=I_H R_a,
\]
since both sides evaluate $\varphi^\sharp$ at
$(a^{-1}x^{-1}g_0,\ldots,a^{-1}x^{-1}g_q)$. Representative-independence
then gives this identity on $L^0$ classes. In the ambient model the map
$R_a$ induces the usual restriction-conjugation map, by naturality of
measurable cohomology.
All these maps act on equivalence classes on the ambient spaces.

If $K$ is finite, compute its cohomology in its own homogeneous complex
$\operatorname{Map}(K^{q+1},A_Y)^K$. The finite average
\[
 (sc)(k_0,\ldots,k_{q-1})
 =\frac1{|K|}\sum_{k\in K}c(k,k_0,\ldots,k_{q-1})
\]
is defined on every cochain and satisfies $ds+sd=1$ in positive
degrees. Thus the right side of \eqref{eq:shapiro} vanishes for $q>0$.
\end{proof}

\subsection{The measurable double complex}
Consider the array
\begin{equation}\label{eq:double-complex}
 \mathcal A^{p,q}=L^0(G^{q+1}\times X^p)^G,
 \qquad p,q\ge0,
\end{equation}
where $X^0$ is a point. The vertical differential $d_G$ deletes a group
variable; the horizontal differential is $(-1)^q\delta_X$, where
$\delta_X$ deletes a projective variable. The sign makes the two
differentials anticommute, so their sum has square zero.
The index $p$ counts projective points, not action-cochain degree:
$p$ points correspond to degree $p-1$.

The vertical-first spectral sequence has first page and differentials
\begin{equation}\label{eq:spectral-sequence}
 E_1^{p,q}=\Hm^q(G;L^0(X^p)),\qquad
 d_r:E_r^{p,q}\longrightarrow E_r^{p+r,q-r+1}.
\end{equation}

\begin{lemma}[Acyclicity of the total complex]\label{lem:row-exact}
Every augmented horizontal row of \eqref{eq:double-complex} is exact.
Consequently the spectral sequence \eqref{eq:spectral-sequence}
converges to zero.
\end{lemma}
\begin{proof}
The diffeomorphism
\[
 (k;h_1,\ldots,h_q;\mathbf y)
 \longmapsto(k,kh_1,\ldots,kh_q;k\mathbf y)
\]
identifies the diagonal $G$-action with translation in $k$.
By \Cref{lem:descent},
$\mathcal A^{p,q}\cong L^0(G^q\times X^p)$; this is a quotient
identification, not evaluation at $g_0=e$. Deletion in $X$ commutes with
it. Thus a row is the augmented deletion complex in the $X$-variables.

Suppose $p\ge1$ and $\delta_X F=0$. For almost every $y\in X$,
Fubini permits slicing the first $X$-variable at $y$ and makes the
sliced cocycle identity valid almost everywhere. Fix such a $y$ and put
$s_yF(\mathbf g;\mathbf x)=F(\mathbf g;y,\mathbf x)$.
The deletion identity gives $F=\delta_Xs_yF$ almost everywhere.
This proves exactness; at $p=0$ it is immediate. The harmless factor
$(-1)^q$ does not affect exactness.

The array is first quadrant and each total degree contains finitely
many bidegrees. Therefore both filtrations of its direct-sum total
complex converge to its cohomology. Taking horizontal cohomology first
gives zero, and hence so does \eqref{eq:spectral-sequence}.
\end{proof}

The bottom row of $E_1$ is the invariant measurable action complex, so
\begin{equation}\label{eq:bottom-row}
 E_2^{p,0}=\Hm^{p-1}(G\curvearrowright X)\quad(p\ge2).
\end{equation}
In particular, the desired group occurs at $(p,q)=(5,0)$.

For $p=1,2,3$, \Cref{prop:stabilizers,lem:shapiro} identify the positive
rows with the cohomology of $Q,L,H$. For $p\ge4$, the product charts in
\Cref{prop:coordinates}, partitioned measurably over the base, give
\[
 \Omega_p\cong\cU_p\times G/Z(G)
\]
as a measure-class $G$-space. Shapiro with coefficient $L^0(\cU_p)$
and finite stabilizer $Z(G)$ gives zero in every positive degree.
Thus, for $q>0$,
\begin{equation}\label{eq:positive-rows}
 \begin{aligned}
 E_1^{0,q}&=\Hc^q(G),& E_1^{1,q}&=\Hc^q(Q),\\
 E_1^{2,q}&=\Hc^q(L),& E_1^{3,q}&=\Hc^q(H),\\
 E_1^{p,q}&=0\quad(p\ge4).
 \end{aligned}
\end{equation}
Austin--Moore identifies the real-coefficient groups in columns zero
through three with continuous cohomology. The higher columns vanish
by finite-group averaging with coefficients in $L^0(\cU_p)$.

\subsection{The degree-four transgression}
Choose the base configurations
\[
 \boldsymbol\xi_0=*,\quad\boldsymbol\xi_1=([e_1]),\quad
 \boldsymbol\xi_2=([e_1],[f_1]),\quad
 \boldsymbol\xi_3=([e_1],[f_1],[e_1+f_1+e_2]),
\]
with stabilizers $K_0=G,K_1=Q,K_2=L,K_3=H$. For each deletion from
$\boldsymbol\xi_p$, choose $a_{p,i}\in G$ with
$\partial_i\boldsymbol\xi_p=a_{p,i}\boldsymbol\xi_{p-1}$.
For $p=1$ choose $a_{1,0}=e$. The face map of homogeneous orbits is
$gK_p\mapsto ga_{p,i}K_{p-1}$, and the subgroup map is
\[
 \jmath_{p,i}:K_p\to K_{p-1},\qquad h\mapsto a_{p,i}^{-1}ha_{p,i}.
\]
By \Cref{lem:shapiro}, including the horizontal sign in
\eqref{eq:double-complex},
\begin{equation}\label{eq:first-differential}
 d_1^{p-1,q}=(-1)^q\sum_{i=0}^{p-1}(-1)^i\jmath_{p,i}^*
 \qquad(1\le p\le3).
\end{equation}
This also holds under the natural comparison with continuous cohomology.
In particular the first map is $(-1)^q\operatorname{res}_Q^G$;
in degree one all three terms of the map from $L$ to $H$ are zero.

By \Cref{lem:stabilizer-cohomology}, the first page in vertical degrees
one through four is
\begin{equation}\label{eq:E1-table}
\begin{array}{c|ccccc}
 &p=0&p=1&p=2&p=3&p\ge4\\
 &G&Q&L&H&\\ \hline
 q=4&0&\RR&\RR&0&0\\
 q=3&\RR&\RR&\RR&0&0\\
 q=2&0&0&0&\RR&0\\
 q=1&0&\RR&\RR&\RR^2&0.
\end{array}
\end{equation}
The three zero face restrictions give
\begin{equation}\label{eq:surviving-term}
 E_2^{3,1}=\Hc^1(H)=\Hom_{\RR}(\CC,\RR).
\end{equation}
There is no outgoing $d_1$, since column four vanishes in positive
degrees. Also
\begin{equation}\label{eq:other-terms-zero}
 E_2^{2,2}=E_2^{1,2}=E_2^{1,3}=E_2^{0,3}=E_2^{0,4}=0.
\end{equation}
The degree-two terms vanish already on $E_1$. The terms $(1,3)$ and
$(0,3)$ vanish because the incoming/outgoing map
$\Hc^3(G)\to\Hc^3(Q)$ is an isomorphism, with a minus sign that
does not change its kernel or image. Finally $\Hc^4(G)=0$.

\begin{proof}[Proof of \Cref{thm:ordinary}]
The only possible incoming differentials to the bottom term $(5,0)$ are
\[
 (3,1)\xrightarrow{d_2}(5,0),\quad
 (2,2)\xrightarrow{d_3}(5,0),\quad
 (1,3)\xrightarrow{d_4}(5,0),\quad
 (0,4)\xrightarrow{d_5}(5,0).
\]
The last three sources vanish by \eqref{eq:other-terms-zero}. There
is no outgoing differential from $(5,0)$ on pages $r\ge2$. For the source $(3,1)$, the only
possible incoming differentials on pages $r\ge2$ come from $(1,2)$ and
$(0,3)$, which also vanish. After $d_2$ there is no outgoing differential
from $(3,1)$, since the target would have negative vertical degree.

Consequently the eventual term at $(3,1)$ is the kernel of
$d_2^{3,1}$, and that at $(5,0)$ is its cokernel. Both must be zero
by \Cref{lem:row-exact}. Therefore
\[
 d_2^{3,1}:\Hom_{\RR}(\CC,\RR)
 \xrightarrow{\ \cong\ }\Hm^4(G\curvearrowright X;\RR).
\]
This is the required transgression isomorphism.
\end{proof}

In \Cref{sec:cocycle}, we construct a rational cocycle whose real
components represent a basis. Their independence follows from the
period calculation in \Cref{sec:completion}.

\section{A rational cocycle and finite orbit cycles}\label{sec:cocycle}

We construct a rational measurable four-cocycle and a family of finite
orbit cycles with uniformly bounded $\ell^1$-mass. Their periods give
the obstruction to bounded representatives.

\subsection{Affine-coordinate differences}
For lifts $v_i$ of generic projective points, set
$a_{ij}=\omega(v_i,v_j)$ and write
\[
 [ijkl]=a_{ij}a_{kl}-a_{ik}a_{jl}+a_{il}a_{jk}
\]
for the Pfaffian in the \emph{displayed} order. For five distinct indices
define
\begin{equation}\label{eq:affine-difference}
 r_{ijk}(l,m)=-\frac{a_{ik}a_{jk}[ijlm]}{[ijkl][ijkm]}.
\end{equation}
All denominators are nonzero on generic configurations.

\begin{lemma}\label{lem:affine-difference}
The quantity $r_{ijk}(l,m)$ is independent of the lifts and
$G$-invariant. For a fixed ordered triple $(i,j,k)$ it satisfies
\begin{equation}\label{eq:additivity}
 r_{ijk}(l,m)=-r_{ijk}(m,l),\qquad
 r_{ijk}(l,m)+r_{ijk}(m,n)=r_{ijk}(l,n).
\end{equation}
\end{lemma}
\begin{proof}
Under $v_s\mapsto\lambda_sv_s$, the numerator and denominator of
\eqref{eq:affine-difference} both acquire the factor
$\lambda_i^2\lambda_j^2\lambda_k^2\lambda_l\lambda_m$.
This proves lift-independence; invariance follows from preservation of
$\omega$.

Project along $\langle v_i,v_j\rangle$ onto
$\langle v_i,v_j\rangle^\perp$, a two-dimensional symplectic space.
Denote the projected vectors by $\bar v_p$. A direct projection
calculation gives
\[
 \omega(\bar v_p,\bar v_q)=\frac{[ijpq]}{a_{ij}}.
\]
Choose $e=\bar v_k$ and $f$ with $\omega(e,f)=1$. For
$p\notin\{i,j,k\}$, write $\bar v_p=c_pe+d_pf$. Then
$d_p=[ijkp]/a_{ij}\ne0$ by genericity, so $t_p=c_p/d_p$ is defined.
Changing $f$ adds the same constant to these affine coordinates. Moreover,
\[
 t_l-t_m=\frac{a_{ij}[ijlm]}{[ijkl][ijkm]},\qquad
 r_{ijk}(l,m)=-\frac{a_{ik}a_{jk}}{a_{ij}}(t_l-t_m).
\]
Equation \eqref{eq:additivity} follows.
\end{proof}

Define a complex-valued function on generic five-tuples by alternation:
\begin{equation}\label{eq:R}
 \cR(\ell_0,\ldots,\ell_4)
 =\frac1{5!}\sum_{\sigma\in S_5}\sgn(\sigma)
 r_{\sigma(0)\sigma(1)\sigma(2)}(\sigma(3),\sigma(4)).
\end{equation}
It is rational, invariant, and alternating, hence continuous on the
generic locus.

\begin{proposition}\label{prop:cocycle}
The function $\cR$ satisfies $d\cR=0$ pointwise on generic six-tuples.
It therefore defines a complex measurable action four-cocycle, and every
$\lambda\in\Hom_{\RR}(\CC,\RR)$ defines a real class
$[\lambda\circ\cR]\in\Hm^4(G\curvearrowright X;\RR)$.
\end{proposition}
\begin{proof}
Expand $d\cR$ and group terms by the ordered triple $(i,j,k)$.
If $l,m,n$ are the remaining three labels, their contribution, up to
a common factor and sign, is
\[
 r_{ijk}(l,m)-r_{ijk}(l,n)+r_{ijk}(m,n)=0.
\]
This proves the cocycle identity. Extending $\cR$ by zero on the
nongeneric locus gives a measurable cochain. Deletion is
nonsingular, so the pointwise identity on the generic set gives the
almost-everywhere identity in the action complex.
\end{proof}

For explicit evaluation, the $120$ summands can be grouped into $30$.
For a five-point Gram matrix let $P_s=\Pf(A_{\widehat s})$, with
remaining indices in increasing order. Then
\begin{equation}\label{eq:thirty}
 \cR(A)=\frac1{30}\sum_{k=0}^4
 \sum_{\{i,j\}\subset\{0,\ldots,4\}\setminus\{k\}}
 (-1)^{i+j}a_{ik}a_{jk}\frac{P_k}{P_lP_m},
\end{equation}
where $\{l,m\}$ is the complementary pair. Indeed, replacing ordered
Pfaffians in \eqref{eq:R} by the $P_s$ leaves the sign $(-1)^{i+j}$.
Interchanging $i,j$ or $l,m$ leaves the resulting summand unchanged,
so each occurs four times. Formula \eqref{eq:thirty} is the finite
rational expression used in the appendix.

\subsection{Configuration chains and orbit coinvariants}
A configuration $k$-simplex is a generic ordered $(k+1)$-tuple.
Let $\tuplechain_k$ be the real vector space generated by generic ordered
$(k+1)$-tuples, with the orientation relation
$[\sigma\mathbf x]=\sgn(\sigma)[\mathbf x]$. Let
\begin{equation}\label{eq:chain-quotient}
 \cochain_k^{\orb}=\tuplechain_k/
       \langle gc-c:g\in G,\ c\in\tuplechain_k\rangle,
 \qquad q_k:\tuplechain_k\to\cochain_k^{\orb}.
\end{equation}
We call $\tuplechain_k$ the space of alternating configuration chains
and its $G$-coinvariant quotient $\cochain_k^{\orb}$ the space of
\emph{orbit chains}. Their boundary is
\begin{equation}\label{eq:chain-boundary}
 \partial[x_0,\ldots,x_k]
 =\sum_{i=0}^k(-1)^i[x_0,\ldots,\widehat{x_i},\ldots,x_k],
\end{equation}
and $q$ commutes with $\partial$. A finite orbit four-cycle is an
 element $Z\in\cochain_4^{\orb}$ with $\partial Z=0$. Its $\ell^1$-mass is
\begin{equation}\label{eq:mass}
 \|Z\|_1=\inf\left\{\sum_j|c_j|:Z=\sum_jc_j[\mathbf x_j]\right\}.
\end{equation}

An \emph{unoriented orbit} is an orbit under $G$ and vertex
permutations, with permutation signs forgotten. If a tuple is
$G$-equivalent to an odd permutation of itself, we say that the orbit
admits an \emph{odd self-equivalence}. Its generator in
$\cochain_k^{\orb}$ equals its negative and is therefore zero.
Every other unoriented orbit contributes one real basis vector after an
orientation is chosen. The relations do not mix distinct unoriented
orbits, so \eqref{eq:mass} is the $\ell^1$-norm in this basis and
the infimum is attained. A pointwise invariant alternating cochain $c$
pairs with orbit chains by
\[
 \langle c,\textstyle\sum_jc_j[\mathbf x_j]\rangle
       =\sum_jc_jc(\mathbf x_j),
 \qquad
 \langle dc,Z\rangle=\langle c,\partial Z\rangle.
\]
The pairing vanishes on orbits admitting odd self-equivalences.
For a cocycle $c$ and a cycle $Z$, $\langle c,Z\rangle$ is their
\emph{period}. Pointwise evaluation does not define a pairing with
arbitrary measurable equivalence classes.

\subsection{Construction of finite orbit cycles}
Write $[z]$ for the oriented four-point orbit with coordinate
$z\in\cU_4$. The face table in the proof of \Cref{lem:dilation}
gives an identity of real orbit chains:
\begin{equation}\label{eq:special-boundary}
 \partial[\Phi_j(z)]=2[z]-[T_jz],\qquad j=1,2.
\end{equation}
The faces fixed by odd permutations vanish in the coinvariant space,
and the remaining odd symmetry contributes a minus sign. Consequently
\begin{equation}\label{eq:edge}
 \mathfrak e(z)=[\Phi_2(z)]-[\Phi_1(z)],\qquad
 \partial\mathfrak e(z)=[T_1z]-[T_2z].
\end{equation}
Thus $\mathfrak e(z)$ is a degree-four chain whose boundary is the
difference of the degree-three generators at the endpoints of a
$U$-step.

Let $(x,y)=(X,-S^2)$. The two parameters
$z_1=(-X,S)$ and $z_2=(XS,S)$ give consecutive $U$-steps, since
\[
 T_1z_1=(X,-S^2),\quad T_2z_1=T_1z_2,
 \quad T_2z_2=(XS^2,-S^2)=(-xy,y).
\]
Using the same square root in both steps gives $(-xy,y)$, rather
than $S_1(x,y)=(xy,y)$ in \Cref{lem:correction}. Composition with
the even symmetry $\sC$ preserves the oriented orbit class and gives
\begin{equation}\label{eq:M}
 \sM(x,y)=\sC(-xy,y)=(-1/y,x),\qquad \sM^4=\operatorname{id}.
\end{equation}
Four successive pairs of steps therefore close up to an orbit cycle.

Starting at $w_0=(-A^2,-B^2)$, with $A,B\ne0$, the successive
points $w_j=(x_j,y_j)$ and square roots are
\begin{equation}\label{eq:cycle-states}
\begin{array}{c|c|c}
 j&w_j&\text{root }S_j\text{ with }y_j=-S_j^2\\ \hline
 0&(-A^2,-B^2)&B\\
 1&(B^{-2},-A^2)&A\\
 2&(A^{-2},B^{-2})&i/B\\
 3&(-B^2,A^{-2})&i/A.
\end{array}
\end{equation}
Since $w_4=w_0$, the eight chains
\eqref{eq:edge} defined by these four pairs of steps sum to a cycle
$Z(A,B)$. Its boundary telescopes, so
\begin{equation}\label{eq:cycle-mass}
 \partial Z(A,B)=0,\qquad\|Z(A,B)\|_1\le16,
\end{equation}
whenever the displayed configurations are generic.

To write the period, define
\begin{align}
 F_j(a,b)&=\cR(\Phi_j(a,b)),&
 \mathcal G(a,b)&=F_2(a,b)-F_1(a,b),\label{eq:F-G}\\
 \mathcal Q(X,S)&=\mathcal G(-X,S)+\mathcal G(XS,S).\nonumber
\end{align}
Then
\begin{equation}\label{eq:period}
 \begin{split}
 \langle\cR,Z(A,B)\rangle=P(A,B)
 ={}&\mathcal Q(-A^2,B)+\mathcal Q(B^{-2},A)\\
   &+\mathcal Q(A^{-2},i/B)+\mathcal Q(-B^2,i/A).
 \end{split}
\end{equation}
The period is rational in $A,B$. We use the punctured neighborhood of
$(A,B)=(2,3)$ established below.

\begin{lemma}[Genericity of the period family]\label{lem:generic-line}
There is $\delta>0$ such that $Z(A,3)$ is a well-defined generic orbit
cycle whenever $0<|A-2|<\delta$.
\end{lemma}
\begin{proof}
The eight correspondence parameters in \eqref{eq:period} are
\begin{equation}\label{eq:eight-pairs}
 \begin{gathered}
 (A^2,3),\quad(-3A^2,3),\quad(-1/9,A),\quad(A/9,A),\\
 (-A^{-2},i/3),\quad(i/(3A^2),i/3),\quad
 (9,i/A),\quad(-9i/A,i/A).
 \end{gathered}
\end{equation}
For each pair $(a(A),b(A))$, substitute into $\Delta_U$ from
\eqref{eq:discriminant}. The result is a rational function of $A$.
None is identically zero: at $A=3$ the first four have $b=3$ and
$a\in\{9,-27,-1/9,1/3\}$; the last four have $b=i/3$ and
$a\in\{-1/9,i/27,9,-3i\}$. For these nonzero $b$, neither $b-2$
nor $2b-1$ vanishes. The remaining excluded $a$-values are
\[
 \{b+1,b^2-1,(1-b^2)/b\}
 =\begin{cases}
 \{4,8,-8/3\},&b=3,\\
 \{1+i/3,-10/9,-10i/3\},&b=i/3.
 \end{cases}
\]
None occurs in the corresponding list. Thus all eight substituted admissibility functions
are nonzero at $A=3$. A nonzero rational function of one variable has
only finitely many zeros and poles. Choose a disk about $2$ avoiding
all of them except possibly $2$ itself, and also avoiding $0$.
All eight correspondence parameters, and hence all sixteen five-tuples
and their faces,
are generic throughout the punctured disk.
\end{proof}

\section{Averaging of orbit cycles}\label{sec:period-test}

To compare periods with measurable cohomology, we replace finite orbit
cycles by absolutely continuous configuration measure cycles with
controlled total variation.

\begin{samepage}
\begin{theorem}[Measurable period estimate]\label{thm:period-test}
Let $C$ be a real-valued invariant alternating function, continuous on
$\Omega_5$, with $dC=0$ pointwise on $\Omega_6$. Suppose
\begin{equation}\label{eq:bounded-representative}
 C=F+df\quad\text{almost everywhere},
\end{equation}
where $F\in L^\infty(X^5)^G$ and $f\in L^0(X^4)^G$. Then every
finite generic real orbit four-cycle $Z$ satisfies
\begin{equation}\label{eq:period-bound}
 |\langle C,Z\rangle|\le20\|Z\|_1\|F\|_\infty.
\end{equation}
\end{theorem}
\end{samepage}

The construction uses simplices with two variable vertices and a fixed
generic triple. The associated maps to $\cU_5$ are submersions.
Averaging a homologous family of such cycles yields the measure cycle
used in the estimate.

\subsection{The fixed-triple configuration map}
For a fixed generic ordered triple $\mathbf z=(z_0,z_1,z_2)$, let
\[
 U_{\mathbf z}=\{(x,y)\in X^2:(x,y,z_0,z_1,z_2)\in\Omega_5\}.
\]
This is a nonempty open set, and the quotient map on it is
\begin{equation}\label{eq:two-free-map}
 \beta_{\mathbf z}:U_{\mathbf z}\to\cU_5,
 \qquad (x,y)\mapsto\gamma_5(x,y,z_0,z_1,z_2).
\end{equation}

\begin{lemma}\label{lem:two-free}
The map $\beta_{\mathbf z}$ is surjective and is a smooth submersion
at every point of its domain.
\end{lemma}
\begin{proof}
Any generic five-tuple has a generic last triple. By
\Cref{prop:stabilizers}, some element of $G$ moves that triple to
$\mathbf z$. This proves surjectivity.

Let $p:\Omega_5\to\Omega_3$ be the projection onto the last triple. It is the
restriction of a coordinate projection to an open set, so it is a
submersion. Its fibre over $\mathbf z$ is $U_{\mathbf z}$, with the
first two vertices as coordinates, and its fibre tangent space is
$\ker dp$. Given a tangent vector
$\eta\in T_{\gamma_5(\mathbf v)}\cU_5$, lift it to
$w\in T_{\mathbf v}\Omega_5$ by the submersion
$\gamma_5$ of \Cref{prop:coordinates}. For the transitive smooth Lie
group action on $\Omega_3$, the orbit map is a submersion. Thus there
is $\xi\in\mathfrak g$ whose infinitesimal action on $p(\mathbf v)$
is $dp(w)$. Set
\[
 w'=w-\xi_{\Omega_5}(\mathbf v).
\]
Then $dp(w')=0$, so $w'$ is tangent to the fixed-triple fibre.
Invariance of $\gamma_5$ gives $d\gamma_5(w')=\eta$. Thus the
restriction to that fibre is submersive.
\end{proof}

\subsection{Cone identities and face pairings}
For a fixed point $y$, the cone on a tuple is
$s_y(t_0,\ldots,t_k)=(y,t_0,\ldots,t_k)$, whenever the new tuple is
generic. In the configuration-chain complex,
\begin{equation}\label{eq:cone}
 \partial s_y+s_y\partial=1.
\end{equation}
The identity follows by deletion. In degree zero, use the augmentation
$\tuplechain_{-1}=\RR$, $\partial[x]=1$, $s_y(1)=[y]$.

The equivariance identity $g s_y=s_{gy}g$ requires transport of the
apex. Accordingly, all cone operations below are performed on specified
configuration representatives before applying $q_k$, and paired faces
receive transported cone points.

Choose a presentation of a nonzero orbit cycle realizing its $\ell^1$-mass
\begin{equation}\label{eq:cycle-presentation}
 Z=\sum_\sigma\varepsilon_\sigma w_\sigma[\sigma],
 \qquad w_\sigma>0,\quad\varepsilon_\sigma\in\{\pm1\},\quad
 \sum_\sigma w_\sigma=\|Z\|_1.
\end{equation}
Such a presentation exists by the oriented-basis description following
\eqref{eq:mass}. The notation $\sigma$ here denotes an ordered
five-tuple. Its $i$th face is $\tau_{\sigma i}=\partial_i\sigma$,
with incidence sign $\varepsilon_\sigma(-1)^i$.

\begin{lemma}[Pairing weighted faces]\label{lem:face-pairing}
Each face occurrence $(\sigma,i)$ can be split into finitely many
weighted copies with positive weights $r_e$ summing to $w_\sigma$.
Set $\tau_e=\tau_{\sigma i}$ and
$a_e=\varepsilon_\sigma(-1)^i$ for such a copy. The copies can be
paired, with equal weights in each pair, so that
\begin{equation}\label{eq:face-pair}
 \tau_{e'}=g_e(\pi_e\tau_e),\qquad
 a_{e'}\sgn(\pi_e)=-a_e,
\end{equation}
for some $g_e\in G$ and $\pi_e\in S_4$.
\end{lemma}
\begin{proof}
Group face occurrences by their unoriented orbit. On an orbit with no odd self-equivalence,
choose an orientation. Because $\partial Z=0$, the total positive and
negative oriented weights agree. Match them by repeatedly taking the
smaller of two remaining weights and subtracting it from both lists.
This finite procedure gives equal-weight copies with opposite oriented
signs, hence \eqref{eq:face-pair}.

For a face whose orbit admits an odd self-equivalence, split its weight
into two equal copies and pair them using that self-equivalence. Their incidence
signs are equal, and the odd permutation supplies the minus sign in
\eqref{eq:face-pair}. Only face weights are split; no subdivision of the parent simplices is
involved.
\end{proof}

Orient each pair $e\to e'$. Introduce a variable $y_e\in X$ at its
source and set $y_{e'}=g_e y_e$ at the target. Introduce a further
variable $x_\sigma\in X$ for each parent five-tuple. The source
variables $y_e$ and the variables $x_\sigma$ are independent coordinates
on a product of copies of $X$. Restrict to choices for which all tuples
in the cone formulas are generic.

\begin{proposition}[The two-cone identity]\label{prop:two-cone}
There are orbit chains $S Z\in\cochain_4^{\orb}$ and
$K Z\in\cochain_5^{\orb}$, depending on these choices, such that
\begin{equation}\label{eq:two-cone-identities}
 \partial S Z=0,\qquad Z-S Z=\partial K Z.
\end{equation}
The chain $S Z$ has a presentation of mass at most $20\|Z\|_1$.
Every elementary simplex of that presentation consists of two free
vertices and a fixed generic triple. If $C$ is a pointwise $G$-invariant
alternating function on $\Omega_5$ with $dC=0$ on $\Omega_6$, then
$\langle C,S Z\rangle=\langle C,Z\rangle$.
\end{proposition}
\begin{proof}
For each parent simplex form the following alternating configuration chains:
\begin{equation}\label{eq:first-cones}
 \begin{aligned}
 H_\sigma&=\sum_{i=0}^4\ 
       \sum_{e\text{ in }(\sigma,i)}a_er_e\,s_{y_e}\tau_{\sigma i},\\
 B_\sigma&=\sum_{i=0}^4\ 
       \sum_{e\text{ in }(\sigma,i)}a_er_e\,s_{y_e}\partial\tau_{\sigma i}.
 \end{aligned}
\end{equation}
Here $H_\sigma$ has degree four and $B_\sigma$ degree three.
Since the weights of the copies of each face sum to
$w_\sigma$, \eqref{eq:cone} gives
\begin{equation}\label{eq:first-cone-boundaries}
 \partial H_\sigma
 =\varepsilon_\sigma w_\sigma\partial\sigma-B_\sigma,
 \qquad
 \partial B_\sigma
 =\varepsilon_\sigma w_\sigma\partial^2\sigma=0.
\end{equation}
These identities hold before taking orbit quotients.

For a paired face, extending $\pi_e$ by fixing the new first vertex
and transporting the apex gives
\[
 q_4(s_{y_{e'}}\tau_{e'})
 =\sgn(\pi_e)q_4(s_{y_e}\tau_e).
\]
The equal weights and \eqref{eq:face-pair} therefore imply
$\sum_\sigma q_4(H_\sigma)=0$. Taking its boundary, using
\eqref{eq:first-cone-boundaries} and $\partial Z=0$, also gives
\begin{equation}\label{eq:paired-cancellation}
 \sum_\sigma q_4(H_\sigma)=0,
 \qquad\sum_\sigma q_3(B_\sigma)=0.
\end{equation}
The cancellations in \eqref{eq:paired-cancellation} take place in the
coinvariant complex.

Now put $A_\sigma=\varepsilon_\sigma w_\sigma\sigma-H_\sigma$.
Equation \eqref{eq:first-cone-boundaries} says
$\partial A_\sigma=B_\sigma$ and $\partial B_\sigma=0$ in the
configuration-chain complex. Define
\begin{equation}\label{eq:second-cones}
 S Z=\sum_\sigma q_4(s_{x_\sigma}B_\sigma),\qquad
 K Z=\sum_\sigma q_5(s_{x_\sigma}A_\sigma).
\end{equation}
The cone identity and \eqref{eq:paired-cancellation} give
\[
 \partial S Z=\sum_\sigma q_3(B_\sigma)=0,
 \qquad
 \partial K Z=Z-\sum_\sigma q_4(H_\sigma)-S Z=Z-S Z.
\]
Each parent simplex has five faces, each face boundary has four terms,
and the weights of its copies sum to $w_\sigma$. Thus the elementary
presentation of $s_{x_\sigma}B_\sigma$ has mass at most $20w_\sigma$.
Its terms are $(x_\sigma,y_e,t_0,t_1,t_2)$, with the last three
vertices fixed, or the same expression with $y_e=g_e y$.
Finally \eqref{eq:two-cone-identities} and $dC=0$ give
$\langle C,Z-S Z\rangle=\langle dC,K Z\rangle=0$.
\end{proof}

\subsection{Simultaneous genericity of cone parameters}
Let $\mathcal P_0$ be the product of one copy of $X$ for each
$x_\sigma$ and one for each source variable $y_e$. Target variables
are determined by transport, not chosen independently.
There is a nonempty Zariski-open set
$\mathcal P^\circ\subset\mathcal P_0$ on which every five-tuple in
$S Z$ and every six-tuple in $K Z$ is generic. Indeed, each such
condition is an open algebraic condition and is individually nonempty:
a new vertex can be chosen outside finitely many pairing hyperplanes
and spans of old triples. For a term with two new vertices, choose them
successively. Fixed group transformations of a free variable preserve
this conclusion. There are only finitely many conditions, and their
intersection is nonempty because the product of projective spaces
$\mathcal P_0$ is irreducible.

On this set write the elementary presentation as
\begin{equation}\label{eq:parameterized-cycle}
 S Z(\mathbf p)=\sum_\alpha c_\alpha[p_\alpha(\mathbf p)],
 \qquad p_\alpha:\mathcal P^\circ\to\cU_5,
 \qquad\sum_\alpha|c_\alpha|\le20\|Z\|_1.
\end{equation}
Every $p_\alpha$ is a submersion: it is the projection onto the two
independent free variables of that term, followed by fixed
transformations and the map of \Cref{lem:two-free}. Restriction to the open set $\mathcal P^\circ$ preserves
submersivity.

\subsection{Configuration measure chains and duality}
Let
$\mathfrak M_k$ be the space of finite signed Borel measures on
$\cU_{k+1}$ for $k=3,4,5$, with total variation norm
$\|\nu\|_{\TV}=|\nu|(\cU_{k+1})$. Define boundaries $\partial_k:\mathfrak M_k\to\mathfrak M_{k-1}$
for $k=4,5$ and alternating projections by
\begin{equation}\label{eq:measure-operators}
 \begin{aligned}
 \partial_k\nu&=\sum_{i=0}^{k}(-1)^i(\partial_i)_*\nu
       &&(k=4,5),\\
 \mathbf A_k\nu&=\frac1{(k+1)!}
                 \sum_{\sigma\in S_{k+1}}\sgn(\sigma)\sigma_*\nu
       &&(k=3,4,5).
 \end{aligned}
\end{equation}
We suppress the boundary subscript when its domain is clear.
Here permutations act on normalized Gram coordinates. The operator
$\mathbf A_k$ is the alternating projection on measures and does not
increase total variation. Absolute continuity means absolute continuity of the total variation
with respect to the smooth measure class.

The face and permutation identities give
\begin{equation}\label{eq:measure-identities}
 \begin{gathered}
 \partial_4\partial_5=0,\qquad
 \partial_k\mathbf A_k=\mathbf A_{k-1}\partial_k\quad(k=4,5),\\
 \jmath_k([\mathbf x])=\mathbf A_k\delta_{\gamma_{k+1}(\mathbf x)}
       \quad(k=3,4,5).
 \end{gathered}
\end{equation}
The last formula defines norm-nonincreasing maps
$\jmath_k:\cochain_k^{\orb}\to\mathfrak M_k$, with
$\partial_k\jmath_k=\jmath_{k-1}\partial$ for $k=4,5$. To verify these
statements, pair measures with bounded Borel functions. Pushforward is
adjoint to pullback, and the assertions reduce to $d^2=0$,
$d\Alt=\Alt d$, and the orbit and orientation relations.
Bounded Borel functions separate finite measures, so these are actual
measure identities. For an alternating cochain $C$,
\begin{equation}\label{eq:alternating-pairing}
 \int C\,d\mathbf A_k\nu=\int C\,d\nu,
 \qquad \int C\,d\jmath_k(Z)=\langle C,Z\rangle,
\end{equation}
whenever the integrals exist. Faces and permutations are nonsingular;
therefore their pushforwards preserve absolute continuity.

For a measurable cochain $f$ and a finite signed measure $\nu$, assume
$f\circ\partial_i\in L^1(|\nu|)$ for every face. Then boundary--coboundary
duality gives
\begin{equation}\label{eq:Stokes}
 \int df\,d\nu
 =\sum_i(-1)^i\int f\circ\partial_i\,d\nu
 =\int f\,d\partial\nu.
\end{equation}
Indeed, $|(\partial_i)_*\nu|\le(\partial_i)_*|\nu|$ ensures all
terms are integrable, and the equality is the definition of pushforward.
In particular the integral is zero when $\partial\nu=0$.

\begin{proof}[Proof of \Cref{thm:period-test}]
If $Z=0$ there is nothing to prove. Apply alternation to
\eqref{eq:bounded-representative}; this replaces $F,f$ by alternating
cochains and does not increase $\|F\|_\infty$. Descend them to
$\cU_5,\cU_4$ using \Cref{prop:coordinates}. The pointwise invariant
function $C$ also descends, and its descended function is continuous
by the local sections of $\gamma_5$. Choose the presentation,
face pairing, and parameterized cycle above. Fix a coordinate box
$P\Subset\mathcal P^\circ$ and a nonnegative smooth compactly
supported density $\rho(\mathbf p)\,d\mathbf p$ in $P$ with positive
integral.

For every elementary term, every permutation, and every face, the map
$\partial_i\circ\sigma\circ p_\alpha:P\to\cU_4$ is nonsingular.
Hence the function
\begin{equation}\label{eq:weight}
 W(\mathbf p)=\sum_\alpha\sum_{\sigma\in S_5}\sum_{i=0}^4
 \left|f\bigl(\partial_i\sigma p_\alpha(\mathbf p)\bigr)\right|
\end{equation}
is finite almost everywhere. Choose the probability measure
\begin{equation}\label{eq:mu}
 d\mu(\mathbf p)=\frac1{Z_0}
       \frac{\rho(\mathbf p)}{1+W(\mathbf p)}\,d\mathbf p,
 \qquad
 Z_0=\int_P\frac{\rho}{1+W}\,d\mathbf p>0.
\end{equation}
Since each summand in $W$ is bounded above by $1+W$, every face
evaluation in \eqref{eq:weight} is $\mu$-integrable.

Define the averaged measure cycle by finite pushforwards:
\begin{equation}\label{eq:averaged-cycle}
 \nu=\sum_\alpha c_\alpha\,
             \mathbf A_4(p_\alpha)_*\mu.
\end{equation}
Since every $p_\alpha$ is a submersion, these pushforwards are
absolutely continuous, by Fubini in submersion charts. So is their
alternation. Moreover,
\begin{equation}\label{eq:averaged-mass}
 \|\nu\|_{\TV}\le\sum_\alpha|c_\alpha|\le20\|Z\|_1.
\end{equation}
For each $\mathbf p$, \Cref{prop:two-cone} and the chain map
$\jmath$ give
\[
 \partial\left(\sum_\alpha c_\alpha
       \mathbf A_4\delta_{p_\alpha(\mathbf p)}\right)=0.
\]
Integrate this identity against $\mu$, interpreting integration as the
finite sum of pushforwards in \eqref{eq:averaged-cycle}. Boundary is
also a finite sum of pushforwards, so $\partial\nu=0$ as an equality
of signed measures. Equivalently, the identity holds against every
bounded Borel test function.

For the same reason the pointwise period identity in
\Cref{prop:two-cone} integrates to
\begin{equation}\label{eq:averaged-period}
 \int C\,d\nu=\langle C,Z\rangle.
\end{equation}
The sets $\sigma p_\alpha(\operatorname{supp}\mu)$ are compact in
$\cU_5$. Continuity of $C$ therefore gives $C\in L^1(|\nu|)$.

For $f$ we use the variation bound
\[
 |\nu|\le\frac1{5!}\sum_\alpha\sum_{\sigma\in S_5}
              |c_\alpha|(\sigma\circ p_\alpha)_*\mu.
\]
Equations \eqref{eq:weight}--\eqref{eq:mu} therefore imply
$f\circ\partial_i\in L^1(|\nu|)$ for each face. By
\eqref{eq:Stokes}, $\int df\,d\nu=0$. Finally
\eqref{eq:bounded-representative} holds $\nu$-almost everywhere because
$\nu$ is absolutely continuous. Combining this with
\eqref{eq:averaged-period} and \eqref{eq:averaged-mass} gives
\[
 |\langle C,Z\rangle|=\left|\int F\,d\nu\right|
 \le\|F\|_\infty\|\nu\|_{\TV}
 \le20\|Z\|_1\|F\|_\infty.
\]
\end{proof}

The averaging measure may depend on the primitive, whereas the bound
$20\|Z\|_1$ does not. Thus bounded representatives have uniformly
bounded periods on families of generic orbit cycles with uniformly
bounded $\ell^1$-mass.

\section{Period divergence and the vanishing theorem}\label{sec:completion}

The Laurent expansion of $P(A,3)$ excludes bounded representatives of
all nonzero real components of $\cR$. Together with comparison
injectivity and \Cref{thm:ordinary}, this proves the vanishing theorem.

\begin{proposition}[Laurent expansion of the period]\label{prop:pole}
As $A\to2$ in a sufficiently small punctured complex disk,
\begin{equation}\label{eq:Laurent}
 P(A,3)=\frac{22}{45(A-2)^2}+\frac{61}{180(A-2)}+O(1).
\end{equation}
The cycle $Z(A,3)$ is generic throughout that disk.
\end{proposition}
\begin{proof}
Genericity is \Cref{lem:generic-line}. Insert the two rational
functions $F_1,F_2$ of \eqref{eq:F-G}, given explicitly in
\Cref{app:algebra}, into \eqref{eq:period}. Each summand has at
most a double pole at $A=2$. For the four blocks in that equation,
write $Q_j(A)$ for the block and $h_j(A)=(A-2)^2Q_j(A)$.
Their values and first derivatives at $2$ are
\begin{equation}\label{eq:pole-table}
\begin{array}{c|rr}
 Q_j(A)&h_j(2)&h_j'(2)\\ \hline
 \mathcal Q(-A^2,3)&-1/10&-1/20\\
 \mathcal Q(1/9,A)&53/90&7/18\\
 \mathcal Q(A^{-2},i/3)&0&0\\
 \mathcal Q(-9,i/A)&0&0.
\end{array}
\end{equation}
The appendix gives the functions and the termwise calculation of these
numbers. Summing the columns gives $22/45$ and $61/180$. After these
polar terms are subtracted, the rational function is holomorphic at
$2$. Shrinking the disk to avoid any other pole makes the remainder
bounded, proving \eqref{eq:Laurent}.
\end{proof}

\begin{theorem}[Nonexistence of bounded representatives]\label{thm:no-bounded}
For every nonzero $\lambda\in\Hom_{\RR}(\CC,\RR)$, the class
$[\lambda\circ\cR]\in\Hm^4(G\curvearrowright X;\RR)$ has no
essentially bounded representative. Moreover, the map
\begin{equation}\label{eq:basis}
 \Theta:\Hom_{\RR}(\CC,\RR)\longrightarrow
           \Hm^4(G\curvearrowright X;\RR),
 \qquad\lambda\longmapsto[\lambda\circ\cR]
\end{equation}
is an isomorphism.
\end{theorem}
\begin{proof}
Suppose $\lambda\circ\cR=F+df$ almost everywhere, with $F$ bounded
and $f$ finite measurable, both invariant. By \Cref{thm:period-test}
and \eqref{eq:cycle-mass},
\begin{equation}\label{eq:uniform-P}
 |\lambda(P(A,3))|\le320\|F\|_\infty
\end{equation}
for every sufficiently small $0<|A-2|$.
Choose $\theta$ so that $\lambda(e^{-2i\theta})\ne0$; such a
choice exists because the unit circle spans $\CC$ as a real vector
space. Along $A=2+re^{i\theta}$, \eqref{eq:Laurent} gives
\[
 \lambda(P(2+re^{i\theta},3))
 =\frac{22}{45r^2}\lambda(e^{-2i\theta})+O_\lambda(r^{-1}).
\]
The nonzero double-pole term dominates as $r\downarrow0$, contradicting
\eqref{eq:uniform-P}. 
If $\Theta(\lambda)=0$, then $\lambda\circ\cR$ has the bounded
representative zero. The first assertion therefore forces $\lambda=0$.
So $\Theta$ is injective. Its domain and target both have real
dimension two by \Cref{thm:ordinary}, and it is consequently an
isomorphism. In particular $\operatorname{Re}\cR$ and
$\operatorname{Im}\cR$ represent a basis.
\end{proof}

\begin{proof}[Proof of \Cref{thm:main}]
Let $\alpha\in\Hmb^4(G\curvearrowright X)$. By \Cref{thm:no-bounded},
its ordinary image has the form $[\lambda\circ\cR]$ for a unique
$\lambda\in\Hom_{\RR}(\CC,\RR)$. This image has a bounded
representative, namely any bounded cocycle representing $\alpha$.
The same theorem forces $\lambda=0$. Injectivity of $c_X^4$ from
\Cref{cor:injective} now gives $\alpha=0$.
Equation \eqref{eq:projective-reduction} proves the rank-two group
vanishing.

For the complex symplectic family, rank one is the Bucher--Savini
input and rank two is the result just proved. The injections in
\Cref{prop:reduction} give vanishing in higher rank by induction.
The finite central isogeny $\Sp(4,\CC)\to\SO(5,\CC)$ supplies the
rank-two base for the odd orthogonal family; rank one follows from
$\SL(2,\CC)\to\SO(3,\CC)$. Bounded cohomology is unchanged by
these finite central kernels \cite{MonodBook}, and the odd orthogonal
injections in \Cref{prop:reduction} complete the induction.
\end{proof}

\begin{corollary}[Exactness of the bounded Gram complex]\label{cor:exactness}
For the operators \eqref{eq:D} and \eqref{eq:E},
\[
 \ker(E:L^\infty(\cU_5)\to L^\infty(\cU_6))
 =D(L^\infty(\cU_4)).
\]
Every bounded solution $F$ of $EF=0$ therefore admits a bounded
primitive $f$ with $Df=F$. The range of $D$ is closed.
On alternating cochains, $D$ is bounded below by $1/7$ and has zero
kernel.
\end{corollary}
\begin{proof}
Exactness is \eqref{eq:exact-model} and the vanishing just proved.
Since $E$ is bounded, its kernel is closed. The alternating assertions
are \Cref{thm:bounded-primitives}.
\end{proof}

\appendix
\section{Explicit period computations}\label{app:algebra}

We record the rational functions and Laurent coefficients used in
\eqref{eq:pole-table}.

Let $F_j(a,b)=\cR(\Phi_j(a,b))$. Inserting the five Pfaffians
\eqref{eq:five-pfaffians} into the thirty-term formula
\eqref{eq:thirty}, followed by \eqref{eq:Phi}, gives
\begin{align}
 F_1(a,b)&=\frac{N_1(a,b)}
 {30(b-2)^2(a-b-1)^2(a-b^2+1)},\label{eq:F1}\\
 F_2(a,b)&=-\frac{N_2(a,b)}
 {30(2b-1)^2(a-b-1)^2(ab+b^2-1)},\label{eq:F2}
\end{align}
where
\begin{align*}
 N_1(a,b)={}&3a^4-4a^3b^2-8a^3b-4a^3
 +5a^2b^4-8a^2b^3\\*
 &+48a^2b^2-40a^2b+34a^2
 -8ab^5+8ab^3-20ab^2\\*
 &-24ab-4a+6b^6-8b^5-b^4+16b^3-8b^2+8b+19,\\[3pt]
 N_2(a,b)={}&3a^4b^2-4a^3b^3-8a^3b^2-4a^3b
 +34a^2b^4-40a^2b^3\\*
 &+48a^2b^2-8a^2b+5a^2
 -4ab^5-24ab^4-20ab^3\\*
 &+8ab^2-8a+19b^6+8b^5-8b^4+16b^3-b^2-8b+6.
\end{align*}
These identities involve only rational operations on the entries of a
five-point Gram matrix. Clearing the displayed denominators reduces
each to a polynomial identity.

For the eight pairs $(a_j(A),b_j(A))$ in \eqref{eq:eight-pairs}, set
\[
 g_j(A)=F_2(a_j(A),b_j(A))-F_1(a_j(A),b_j(A)).
\]
Then $P(A,3)=\sum_{j=1}^8g_j(A)$. The pole contributions are as
follows:
\begin{equation}\label{eq:edge-poles}
\begin{array}{c|c|rr}
 j&(a_j(A),b_j(A))&[\,(A-2)^{-2}\,]g_j&[\,(A-2)^{-1}\,]g_j\\ \hline
 1&(A^2,3)&-1/10&-1/20\\
 2&(-3A^2,3)&0&0\\
 3&(-1/9,A)&14/45&1/5\\
 4&(A/9,A)&5/18&17/90\\
 5&(-A^{-2},i/3)&0&0\\
 6&(i/(3A^2),i/3)&0&0\\
 7&(9,i/A)&0&0\\
 8&(-9i/A,i/A)&0&0.
\end{array}
\end{equation}
The brackets in the column headings mean the corresponding Laurent
coefficient. For $j=1$, the only vanishing denominator factor
at $A=2$ is $(a-b-1)^2=(A^2-4)^2$. For $j=3,4$ it is
$(b-2)^2$ in $F_1$. All denominator factors for $j=2,5,6,7,8$ are nonzero at $2$. Thus no higher-order poles occur.

Set $h_j(A)=(A-2)^2g_j(A)$ and remove its removable singularity at
$A=2$. The two coefficients are $h_j(2)$ and $h_j'(2)$. For $j=1$,
\[
 \begin{split}
 h_1(A)={}&-\frac{N_2(A^2,3)}
 {750(A+2)^2(3A^2+8)}
 -\frac{N_1(A^2,3)}{30(A+2)^2(A^2-8)}.
 \end{split}
\]
Substitution and differentiation give $h_1(2)=-1/10$ and
$h_1'(2)=-1/20$. For $j=3,4$ the $F_2$ term is
holomorphic, so it contributes zero to both entries; for the $-F_1$
term, cancel $(b-2)^2=(A-2)^2$ first and then substitute and
differentiate. This gives $(14/45,1/5)$ and $(5/18,17/90)$,
respectively. Grouping consecutive summands yields \eqref{eq:pole-table},
and summing gives
\[
 -\frac1{10}+\frac{14}{45}+\frac5{18}=\frac{22}{45},\qquad
 -\frac1{20}+\frac15+\frac{17}{90}=\frac{61}{180}.
\]
\section{LEAN formulation}
Our results are formalized in LEAN theorem prover, available at http://github.com/wqirocks/H4.Our formalization only includes our work, results from other scholars being used in this article were used as axioms, hence were not formalized by us. 
\section*{Declaration of generative AI and AI-assisted technologies}

During the preparation of this manuscript, the authors used OpenAI's
ChatGPT to assist with research and writing. The authors reviewed and,
where necessary, revised all AI-assisted material included in the
manuscript.

Most of the formalisation work was carried out by ChatGPT. The authors
proposed or developed parts of the framework and intervened as needed to
guide the direction and methods of formalisation. At least two authors
independently checked the formalisation results, with particular attention
to whether the formalised statements, assumptions, and axioms faithfully
represent the intended mathematics.

The authors take responsibility for the original results and conclusions
presented in this article and for the accuracy with which cited results
are stated and applied. Responsibility for the contents of cited
publications remains with their respective authors; this declaration does
not constitute an independent verification or guarantee of the correctness
of those publications. Any conclusion in this article that depends on an
external result relies on the validity of that result.

\end{document}